\documentclass[reqno,11pt]{amsart}

\usepackage{graphicx}
\usepackage {appendix}
\usepackage{amsfonts}
\usepackage{url}
\usepackage{hyperref}
\usepackage{amsmath}
\usepackage{amssymb}
\usepackage{amscd}
\usepackage{color}
\usepackage{amsthm}
\usepackage{bm}
\usepackage{indentfirst}
\usepackage[all, cmtip]{xy}

\usepackage{amstext}
\usepackage{mathrsfs}

\usepackage{color}

\def\E{\ifmmode{\mathbb E}\else{$\mathbb E$}\fi} %natural numbers
\def\N{\ifmmode{\mathbb N}\else{$\mathbb N$}\fi} %natural numbers
\def\R{\ifmmode{\mathbb R}\else{$\mathbb R$}\fi} %real numbers
\def\Q{\ifmmode{\mathbb Q}\else{$\mathbb Q$}\fi} %rational numbers
\def\C{\ifmmode{\mathbb C}\else{$\mathbb C$}\fi} %complex numbers
\def\H{\ifmmode{\mathbb H}\else{$\mathbb H$}\fi} %complex numbers
\def\Z{\ifmmode{\mathbb Z}\else{$\mathbb Z$}\fi} %integers
\def\P{\ifmmode{\mathbb P}\else{$\mathbb P$}\fi} %real numbers
\def\T{\ifmmode{\mathbb T}\else{$\mathbb T$}\fi} %real numbers
\def\SS{\ifmmode{\mathbb S}\else{$\mathbb S$}\fi} %real numbers
\def\DD{\ifmmode{\mathbb D}\else{$\mathbb D$}\fi} %real numbers
\def\R{\ifmmode{\mathbb R}\else{$\mathbb R$}\fi} %rational numbers

\newcommand{\del}{\partial}

\newcommand{\ben}{\begin{enumerate}}
\newcommand{\een}{\end{enumerate}}
\newcommand{\be}{\begin{equation}}
\newcommand{\ee}{\end{equation}}
\newcommand{\bea}{\begin{eqnarray}}
\newcommand{\eea}{\end{eqnarray}}
\newcommand{\bc}{\begin{center}}
\newcommand{\ec}{\end{center}}
\newcommand{\beastar}{\begin{eqnarray*}}
\newcommand{\eeastar}{\end{eqnarray*}}

\theoremstyle{theorem}
\newtheorem{thm}{Theorem}[section]
\newtheorem{cor}[thm]{Corollary}
\newtheorem{lem}[thm]{Lemma}
\newtheorem{prop}[thm]{Proposition}

\newtheorem{conj}[thm]{Conjecture}

\theoremstyle{definition}
\newtheorem{defn}[thm]{Definition}
\newtheorem{rem}[thm]{Remark}
\newtheorem{hypo}[thm]{Hypothesis}

\newtheorem{exm}[thm]{Example}

\numberwithin{equation}{section}

\def\R{{\mathbb R}}

\def\E{{\mathbb E}}
\def\Z{{\mathbb Z}}
\def\C{{\mathbb C}}
\def\R{{\mathbb R}}

\def\N{{\mathbb N}}

\def\SS{{\mathcal S}}

\def\DD{{\mathcal D}}

\def\MM{{\mathcal M}}

\def\11{{\mathbb I}}

\def\delbar{{\overline \partial}}
\def\span{\operatorname{span}}

\def\C{\mathbb{C}}
\def\Z{\mathbb{Z}}

\def\T{\mathbb{T}}

\def\Q{\mathbb{Q}}

\def\E{\ifmmode{\mathbb E}\else{$\mathbb E$}\fi} %natural numbers
\def\N{\ifmmode{\mathbb N}\else{$\mathbb N$}\fi} %natural numbers
\def\R{\ifmmode{\mathbb R}\else{$\mathbb R$}\fi} %real numbers
\def\Q{\ifmmode{\mathbb Q}\else{$\mathbb Q$}\fi} %rational numbers
\def\C{\ifmmode{\mathbb C}\else{$\mathbb C$}\fi} %complex numbers
\def\H{\ifmmode{\mathbb H}\else{$\mathbb H$}\fi} %complex numbers
\def\Z{\ifmmode{\mathbb Z}\else{$\mathbb Z$}\fi} %integers
\def\P{\ifmmode{\mathbb P}\else{$\mathbb P$}\fi} %real numbers
\def\SS{\ifmmode{\mathbb S}\else{$\mathbb S$}\fi} %real numbers
\def\DD{\ifmmode{\mathbb D}\else{$\mathbb D$}\fi} %real numbers

\def\R{{\mathbb R}}

\def\E{{\mathbb E}}
\def\Z{{\mathbb Z}}
\def\C{{\mathbb C}}
\def\R{{\mathbb R}}

\def\N{{\mathbb N}}
\def\MM{{\mathcal M}}

\def\delbar{{\overline \partial}}

  \def\P{\Psi}

\def\CA{{\mathcal A}}

\def\CC{{\mathcal C}}

\def\CF{{\mathcal F}}

\def\CH{{\mathcal H}}

\def\CJ{{\mathcal J}}

\def\CL{{\mathcal L}}
\def\CM{{\mathcal M}}

\def\CS{{\mathcal S}}

\def\CV{{\mathcal V}}
\def\CW{{\mathcal W}}

\def\darr#1{\raise1.5ex\hbox{$\leftrightarrow$}
\mkern-16.5mu #1}

\def\roughly#1{\raise.3ex\hbox{$#1$\kern-.75em
\lower1ex\hbox{$\sim$}}}

\def\opname#1{\mathop{\kern0pt{\rm #1}}\nolimits}

\def\Re{\opname{Re}}
\def\Im{\opname{Im}}

\def\dim{\opname{dim}}

\def\supp{\opname{supp}}

\DeclareMathOperator{\lcs}{\mathrm{lcs}}

\DeclareMathOperator{\Int}{\mathrm{Int}}
\DeclareMathOperator{\Cont}{\mathrm{Cont}}

\newcommand{\norm}[2]{{ \ensuremath{\|} #1 \ensuremath{\|}}_{#2}}

\begin{document}

\quad \vskip1.375truein

\title[Pseudoholomoprhic curves on the $\mathfrak{LCS}$-fication]
{Pseudoholomoprhic curves on the $\mathfrak{LCS}$-fication of contact manifolds}

\author{Yong-Geun Oh, Yasha Savelyev}
\address{Center for Geometry and Physics, Institute for Basic Sciences (IBS), Pohang, Korea \&
Department of Mathematics, POSTECH, Pohang, Korea}
\email{yongoh1@postech.ac.kr}
\address{University of Colima, CUICBAS}
\email{yasha.savelyev@gmail.com}

\begin{abstract}
For each contact diffeomorphism $\phi: (Q,\xi) \to (Q,\xi)$ of $(Q,\xi)$,
we equip its mapping torus $M_\phi$ with a \emph{locally conformal symplectic} form of Banyaga's type,
which we call the \emph{$\lcs$ mapping torus} of the contact diffeomorphism $\phi$.
In the present paper, we consider the product $Q \times S^1= M_{id}$ (corresponding to $\phi = id$)
and develop basic analysis of
the associated $J$-holomorphic curve equation, which has the form
$$
\delbar^\pi w = 0, \quad w^*\lambda\circ j = f^*d\theta
$$
for the map $u = (w,f): \dot \Sigma \to Q \times S^1$
for a $\lambda$-compatible almost complex structure $J$ and a punctured Riemann surface $(\dot \Sigma, j)$.
In particular, $w$ is a \emph{contact instanton} in the sense of [OW2], [OW3].
We develop a scheme of treating the non-vanishing charge by introducing the notion of \emph{charge class}
in $H^1(\dot \Sigma,\Z)$ and develop the geometric framework for the study of pseudoholomorphic
curves, a correct choice of energy and the definition of moduli spaces
towards the construction of compactification of the moduli space on the $\mathfrak{lcs}$-fication
of $(Q,\lambda)$ (more generally on arbitrary locally conformal symplectic manifolds).
\end{abstract}

\thanks{Oh's work is supported by the IBS project \# IBS-R003-D1}
\keywords{locally conformal symplectic manifolds, $\mathfrak{lcs}$-fication of contact manifolds, $\lcs$ instantons}
\maketitle

\def\mq{\mathfrak{q}}
\def\mp{\mathfrak{p}}
\def\mH{\mathfrak{H}}
\def\mh{\mathfrak{h}}
\def\ma{\mathfrak{a}}
\def\ms{\mathfrak{s}}
\def\mm{\mathfrak{m}}
\def\mn{\mathfrak{n}}

\def\Hoch{{\tt Hoch}}
\def\mt{\mathfrak{t}}
\def\ml{\mathfrak{l}}
\def\mT{\mathfrak{T}}
\def\mL{\mathfrak{L}}
\def\mg{\mathfrak{g}}
\def\md{\mathfrak{d}}

\tableofcontents

\section{Introduction}

Symplectic manifolds $(M,\omega)$ have been of much interest in the \emph{global} study of
Hamiltonian dynamics, and symplectic topology via analysis of pseudoholomorphic
curves. In this regard, closedness of the two-form $\omega$ plays an important
role in relation to the dynamics of Hamiltonian diffeomorphisms and the global
analysis of pseudoholomorphic curves, especially the study of compactification of
the relevant moduli spaces.

On the other hand when one takes the
coordinate chart definition of symplectic manifolds and implements the
covariance property of Hamilton's equation, there is no compulsory reason why one should require the
two-form to be closed. Indeed from the point of view of canonical formalism in
Hamiltonian mechanics and construction of the corresponding \emph{bulk physical space}, it is
more natural to require the locally defined canonical symplectic forms
$$
\omega_\alpha = \sum_{i=1}^n dq_i^{\alpha} \wedge dp_i^{\alpha}
$$
to satisfy the cocycle condition
\be\label{eq:cocylce}
\omega_\alpha = \lambda_{\beta\alpha} \omega_\beta, \quad \lambda_{\beta\alpha}\equiv \text{const.}
\ee
with $\lambda_{\gamma\beta}\lambda_{\beta\alpha} = \lambda_{\gamma\alpha}$
as the gluing condition. (See introduction \cite{vaisman:lcs} for a nice
explanation on this point of view) The corresponding bulk constructed in this way
naturally becomes a \emph{locally conformal symplectic manifold}
(abbreviated as l.c.s manifold) whose definition we first recall.

\begin{defn} An $\lcs$ manifold is a triple $(M, \omega, \frak b)$ where $\frak b$ is
a closed one-form, called the Lee form, and $\omega$ is a nondegenerate 2-form satisfying the relation
\be\label{eq:relation}
d ^{\frak b} \omega:=d\omega + \frak b \wedge \omega = 0.
\ee
\end{defn}
(We refer to \eqref{eq:db} for our convention for the definition of the operation $d^{\mathfrak b}$.)

The following definition of morphisms between lcs manifolds
is given in \cite{haller-ryb} with slight variation of phrasing.

\begin{defn}\label{defn:morphism} Let $(M,\omega,\frak b)$ and
$(M^\prime, \omega^\prime, \frak b^\prime)$ be two lcs manifolds.
\begin{enumerate}
\item A diffeomorphism $\phi: M\to M'$ is called {\it lcs} if there exists
$a \in C^\infty(M,\R\setminus\{0\})$ such that
$$
\phi^*\omega^\prime = (1/a) \omega, \quad \phi^*\frak b' = \frak b + d(\ln|a|).
$$
\item An lcs diffeomorphism is \emph{positive} (resp. \emph{negative}) if the function $a$
is positive (resp. negative).
\end{enumerate}
\end{defn}
We refer to \cite{vaisman:lcs}, \cite{haller-ryb}, \cite{banyaga:lcs}, \cite{Banyaga2007}
for a more detailed discussion of general properties of $\lcs$ manifolds and non-trivial examples.
(See also \cite{apost-dlous}, \cite{eliash-murphy} for more recent development concerning
the existence question on the conformal symplectic structure.)

Note that for a positive lcs diffeomorphism, the defining condition can be rewritten as
\be\label{eq:+lcs-definition}
\phi^*\omega^\prime = e^f \omega, \quad \phi^*\frak b' = \frak b - df
\ee
which manifests its similarity to the defining condition $\psi^*\lambda' = e^g \lambda$ of a (nonstrict)
(orientation preserving) contactomorphism $\psi$ with \emph{conformal exponent} $g \in C^\infty(M)$.
(See \cite{oh:contacton-Legendrian-bdy} for the usage of the same terminology used here.)

Locally by choosing $\frak b = d\ell$ for a local function $\ell: U \to \R$
on an open neighborhood $U$, \eqref{eq:relation} is equivalent to
\be\label{eq:local}
d(e^{\ell}\omega) = 0
\ee
and so the local geometry of l.c.s manifold is exactly the same as that of
symplectic manifolds. In particular one can define the notion of Lagrangian
submanifolds, isotropic submanifolds, and coisotropic submanifolds in the same
way as in the symplectic case since the definitions require only nondegeneracy of the two-form $\omega$.

The main purpose of the present paper is to explore the study of $J$-holomorphic curves in an
enlarged bulk of \emph{locally conformal symplectic} manifolds. We will abbreviate
the term locally conformal symplectic as $\lcs$ from now on.
All the local theory of $J$-holomorphic
curves go through without change as in the symplectic case. The main difficulty lies in the global
geometry of $J$-holomorphic curves and it has not been clear
whether Novikov-closedness of $\lcs$ structure $(M,\omega,\frak b)$ would give
meaningful implication to the Fredholm theory of moduli problem
and the study of compactification of $J$-holomorphic
curves with \emph{punctured} Riemann surfaces as their domains.
(See \cite[Introduction]{le-oh:lcs} for such a discussion.)

\subsection{$\lcs$ instantons}

The starting point of the present paper is the second named author's question on  whether
or not contact non-squeezing of Eliashberg-Kim-Polterovich \cite{EKP}
can be generalized to that of $\lcs$ context.
The Eliashberg-Kim-Polterovich contact
non-squeezing theorem as stated
by Fraser ~\cite{fraser} has the following form.
Let $C = \R ^{2(n-1)}  \times S ^{1}  $, $S ^{1} =
\mathbb{R} ^{}/\mathbb{Z}  $,  be the
prequantization space of $\R ^{2n-2} $, or in other
words the contact manifold with the contact form
$d\theta - \lambda$, for $\lambda =
\frac{1}{2}(ydx - xdy)$, (or $\lambda = p\ dq$).
Let $B _{R} $ denote the open radius $R$ ball in
$\mathbb{R} ^{2n-2} $, and $\overline{B} _{R}$ its
topological closure.

To put the current research in some perspective, we recall the following
\emph{contact nonsqueezing theorem} which is the result arising from
a combination of \cite{EKP} and \cite{chiu} (or  \cite{fraser}).
\begin{thm}
   [Eliashberg-Kim-Polterovich~\cite{EKP}, Chiu~\cite{chiu},
   Fraser~\cite{fraser}]
   \label{thm:}
For $R \geq 1$ there is no
contactomorphism, isotopic to the identity, $\phi: C \to C$ so that   $\phi
(\overline{B} _{R} \times S ^{1}) \subset  B _{R}
\times S ^{1} $.
\end{thm}

A Hamiltonian conformal symplectomorphism of an
$\lcs$ manifold $(M,\omega)$, which we just
abbreviate by the short name:
\emph{Hamiltonian lcs map}, is a
$\lcs$ diffeomorphism $\phi _{H} $ generated
analogously to the symplectic case  by a smooth
function $H: M \times [0,1] \to \mathbb{R} $.
Specifically, we define the time dependent vector
field $X _{t} $  by:
\begin{equation*}
\omega (X _{t}, \cdot) = d ^{\frak b} H _{t},
\end{equation*}
for $\frak b$ the Lee form, and then taking $\phi
_{H} $ to be the time 1 flow map of $\{X _{t} \}$.
For example, let $C$ be a manifold with a
contact form $\lambda$, then $\omega = d^{\frak b}
\lambda $ on $C \times S ^{1} $  an $\lcs$
structure called the
$\mathfrak{lcs}$-fication of $(C, \lambda)$, see also the
Definition \ref{def:lcsfication} further ahead.
Then if $\forall t: H _{t} =-1$, then $d ^{\frak b} (H
_{t}) = - \frak b$ and clearly 
$$
X _{t} = R_\lambda \oplus 0,
$$  as a section of $TC
\oplus TS ^{1}$  with $R_\lambda $ the
$\lambda$-Reeb vector field.  (See the list of our sign conventions given in
Convention at the end of this introduction.) Thus in this case the
associated flow is naturally induced by the Reeb
flow.   More generally, given a contact flow on a closed contact
manifold $C$, there is an induced Hamiltonian flow
on the $\mathfrak{lcs}$-fication $C \times S ^{1}$. So that the
following conjecture is a direct generalization of the
contact non-squeezing theorem above.
\begin{conj} \label{conj:lcsnonsqueezing} If $R
   \geq 1$ there is no compactly supported,
   Hamiltonian lcs map $$\phi: \mathbb{R} ^{2n} \times S ^{1} \times S ^{1} \to \mathbb{R} ^{2n} \times S ^{1} \times S ^{1},  $$ so that $\phi (\overline{U} ) \subset U$, for $U := B _{R} \times S ^{1} \times S ^{1}  $ and $\overline{U} $ the topological closure.
  \end{conj}

The way how \cite{EKP} approaches
  the contact non-squeezing problem is to consider open domains of the form $U \times S^1$ and compare their
contact homology which is constructed by considering the \emph{symplectization} $Q \times \R$, where
$Q =\R^{2(n-1)} \times S^1$ is the prequantization space of $\R ^{2(n-1)} $ equipped
with the contact form $d\theta - \lambda$ for $\lambda = \frac{1}{2}(ydx - xdy)$ (or $\lambda = p\ dq$)
and consider the case $U = B_{R}$ where $B_R$ denotes the open radius $R$ ball in $\mathbb{R}^{2(n-1)} $.

In this regard, an interesting general class of $\text{\rm lcs}$ manifolds arises from the mapping torus construction of
contactomorphisms of a contact manifold $(Q,\xi)$ (See Subsection \ref{subsec:mapping-tori} for
a detailed explanation.)

Motivated by this we will more generally associate to each contact manifold $(Q,\xi)$ with $Q = Q^{2n-1}$
an $\lcs$ manifold $Q \times S^1$ equipped with a \emph{${\frak b}$-exact} $\lcs$ form
\be\label{eq:banyaga-lcs-form}
d^{\frak b}\lambda: = \omega_\lambda = d\lambda + {\frak b} \wedge \lambda,
\ee
where $\frak b = \pi^*d \theta$ for  the canonical angular coordinate $\theta$ on $S^1$ with its
period given by $1$. One can check $d\omega_\lambda + \pi^*d\theta \wedge \omega_\lambda =0$, i.e.,
the relation \eqref{eq:relation} holds for $\omega_\lambda$ with the choice of  $\frak b = \pi^*d\theta$.
Since this particular $\lcs$ structure is naturally constructed from the contact manifold $(Q,\lambda)$
it seems reasonable to name it

\begin{defn}[$\mathfrak{lcs}$-fication]
   \label{def:lcsfication} We call the pair $(Q \times S^1, \omega_\lambda)$ an \emph{$\mathfrak{lcs}$-fication}
of the contact manifold $(Q,\lambda)$ where $S^1 = \R/\Z$.
\end{defn}

\begin{rem}
By varying the size of the circle $(S^1,d\theta): = (\R/\Z,[dt])$ to $(S^1_R,d\theta): = (\R/R \Z, [dt])$,
or by taking the one-form $\frak b_R: = \frac{1}{R} \pi^*\theta$ on $S^1 = R/\Z$,
we may regard the $\lcs$-manifold
$$
(Q \times \R, d\lambda + ds \wedge \lambda)
$$
as the limiting case of
$$
(Q \times S^1, d\lambda + \frak b_R \wedge \lambda), \quad \frak b_R = \frac{1}{R} d\theta
$$
as $R \to \infty$. See Subsection \ref{subsec:canonical} for more discussion on
this relationship.
\end{rem}

We observe that
$$
\ker d\lambda = \xi \oplus 0 R_\lambda \oplus 0 \frac{\del}{\del \theta} \cong \xi
$$
for the contact distribution $\xi$ of $Q$. We again denote by $\xi$ for $\ker d\lambda \subset T(Q \times S^1)$
by abuse of notation. Then the tangent bundle of $Q \times S^1$ has the following canonical splitting:
$$
T(Q \times S^1) = \xi \oplus \CV
$$
with
$$
\CV = \span_\R\left\{\frac{\del}{\del \theta}, R_\lambda \right\} =  (\xi)^{\omega_\lambda}
$$
where $(\xi)^{\omega_\lambda}$ denotes the $\omega_\lambda$-orthogonal complement of $\xi$.

As the first step in attacking Conjecture \ref{conj:lcsnonsqueezing} and also as
a first step towards to the study of general pseudoholomorphic curves on general $\lcs$ manifolds,
we study pseudoholomorphic curves on the $\mathfrak{lcs}$-fication $(Q \times S^1, \omega_\lambda)$.
It turns out that the $J$-holomorphic curve on this $\mathfrak{lcs}$-fication is closely related to
that of \emph{contact instanton} studied in \cite{oh-wang:CR-map1,oh-wang:CR-map2,oh:contacton}
for a suitable class of almost complex structures on $Q\times S^1$, which
we call \emph{$\lambda$-admissible}.

\begin{defn}[$\lambda$-admissible almost complex structure]\label{defn:lambda-admissible-J}
We say an almost complex structure $J$ on $M = Q \times S^1$ is \emph{$\lambda$-admissible} if
$J$ satisfies the following:
\begin{enumerate}
\item $J$ is tame with respect to $\omega_\lambda$ in the standard sense,
\item $J$ preserves the splitting \eqref{eq:splitting}, i.e., $J(\xi) = \xi, \,
J(\CV) = \CV$.
\item $J$ satisfies $J \frac{\del}{\del \theta} = R_\lambda$.
\end{enumerate}
We denote by $\CJ(Q \times S^1, \lambda)$ the set of $\lambda$-admissible almost
complex structures.
\end{defn}

\begin{rem} \begin{enumerate}
\item One can weaken the above requirements by dropping the condition (3) which will
provide greater flexibility of the choice of $J$'s. We think such a generalization will be needed for the study of
our starting question on the above mentioned $\lcs$-type nonsqueezing theorem.
We call such $J$ a \emph{$\omega_\lambda$-admissible}
almost complex structure the set of which we denote by $\CJ(Q \times S^1,\omega_\lambda)$.
See Subsection \ref{subsec:mapping-tori} for the case of mapping tori of contact diffeomorphisms.
\item However we leave the study of this general case elsewhere and focus on the
study of pseudoholomorphic curves on the $\mathfrak{lcs}$-fication of contact manifolds in the
present paper.
The study of this special case would be also needed for the computation of
the expected $\lcs$-type invariants which are invariant under the $\omega_\lambda$-compatible
almost complex structures.
\end{enumerate}
\end{rem}

The associated $J$-holomorphic curve equation for the map $u = (w,f): \dot \Sigma \to Q \times S^1$
is reduced to
\be\label{eq:lcs-instanton-intro}
\delbar^\pi w = 0, \quad w^*\lambda \circ j = f^*d\theta.
\ee
(See Proposition \ref{prop:lcs-instanton} for its proof.)
The equation resembles the $J$-holomorphic curve equation in the symplectization $\R \times Q$
except that the \emph{exact} one-form $f^*dr$ in the symplectization case is replaced by
the \emph{closed} one-form $f^*d\theta$. Indeed if we replace $S^1$ by $\R$, our study of
pseudoholomorphic curves on the $\lcs$ manifold equipped with the form
$d\lambda + \frak b \wedge \lambda$ with $\frak b = dr$ reduced to that of the
symplectization of $(Q,\lambda)$.

Furthermore $w$ in $u= (w,f)$ is a \emph{contact instanton}
in the sense of \cite{oh-wang:CR-map1,oh-wang:CR-map2,oh:contacton}, i.e.,
satisfies
\be\label{eq:contacton-intro}
\delbar^\pi w = 0, \quad d(w^*\lambda \circ j) = 0.
\ee
The equation \eqref{eq:lcs-instanton-intro} is augmented by
the datum of its charge given by integrating the specific closed one-form $f^*d\theta$.
To highlight relevance of the similarity of the two equations \eqref{eq:lcs-instanton-intro}, \eqref{eq:contacton-intro},
we give the following definition.
\begin{defn}\label{defn:lcs-instanton-intro}
We call a solution $u=(w,f)$ of \eqref{eq:lcs-instanton-intro} an \emph{$\lcs$ instanton} and
the equation the \emph{$\lcs$ instanton equation}.
\end{defn}

In fact, when by decomposing the closed one-form $f^*d\theta$ into
$$
f^*d\theta = \beta + d\widetilde f
$$
for a harmonic one-form $\beta$ on $\Sigma$
the above equation \eqref{eq:lcs-instanton-intro}
is equivalent to the following
\be\label{eq:abbas-intro}
\begin{cases} \delbar^\pi w = 0,\\
w^*\lambda\circ j - d \widetilde f = \beta\\
\Delta \beta = 0
\end{cases}
\ee
for the triple $(w,\widetilde f, \beta)$ which determines (modulo addition by $\theta_0$) the component
$f: \dot \Sigma \to S^1$ by the equation
$$
\beta + d \widetilde f= f^*d\theta, \quad [\beta] = \eta \text{ in }\, H^1(\dot \Sigma).
$$
(We refer readers to Subsection 4.2 for the detailed discussion on this transformation.)
\begin{rem}
This equation \eqref{eq:abbas-intro} resembles the equation
considered by \cite{ACH} Abbas-Cieliebak-Hofer and by Abbas in \cite{abbas}.
One big difference of our treatment from thereof lies in the case where the domain of the map
$w$ carries punctures:  while \emph{we allow the harmonic one-form $\beta$
not to be smoothly extended across the punctures} (i.e., $Q(r) \neq 0$ in the language of \cite{oh-wang:CR-map1}) at
some puncture $r$), only the case the harmonic one-form extends smoothly
across the puncture (i.e., $Q(r) = 0$ for all punctures $r$) is considered in
both \cite{ACH, abbas}. See below for further discussion on this.
\end{rem}

\subsection{Charge class and asymptotic behavior of $\lcs$ instantons}

The main purpose of the present article is to establish
the two crucial analytical components in the construction of compactification of
the moduli space of solutions of the contact instantons, one the definition of
correct energy and the other the definition of correct (smooth) moduli spaces.
Here what we mean by `correct' choices is that they enable us to define
a compactification that carries all the properties that
are needed in the global study of moduli spaces on the $\lcs$ manifolds $S^1 \times $Q.

One may tempt to use the standard $\omega_\lambda$-energy for $J$-holomorphic curves
(equivalently $\lcs$ instantons),
especially when $Q$ is closed and so $Q \times S^1$ is
a closed $\lcs$ manifold. However in this case the requirement of finite $\omega_\lambda$-energy is
too strong, irrespective whether it is closed or punctured. This is because all
$\lcs$ instantons are completely classified as
stated in Proposition \ref{prop:classify}, when combined with the removal
singularity theorem for the $J$-holomorphic curves with finite harmonic
energy which is the same as $\omega_\lambda$-area.

Therefore to make the  story interesting, one should use different energy for the global study of
moduli space of $\lcs$ instantons. Here we utilize some similarity of
the $\lcs$ instanton equation with the pseudoholomorphic curves in the
symplectization. We will use a variation of Hofer's energy used in the
symplectization in the $\mathfrak{lcs}$-fication of $(Q,\lambda)$. One difficulty
to adapt Hofer-type energy in the $\lcs$ context is that Hofer's energy
used the $\R$-factor of the symplectization $\R \times Q$ or more specifically
the global $r$-coordinate of $\R$ in its definition. It was already pointed out
in \cite{oh-wang:CR-map1} that the case of non-zero charge
$$
Q: = \int_{\{\tau\} \times S^1}  w^*\lambda \circ j
$$
seriously obstructs both the Fredholm theory and the attempt to
construct a compactification of the moduli space of \emph{contact instantons}
(See \cite{oh:contacton} for some effort to overcome these obstructions
in the context of contact instantons.)

In this paper, we are able to overcome both obstructions if the contact instanton $w$ is arising from an
$\lcs$ instanton $u = (w,f)$, i.e., if the charge  form $w^\lambda \circ j$
is specified by the differential of the $S^1$-component $f: \dot \Sigma \to S^1$ of
an $\lcs$ instanton $u = (w,f)$ as in \eqref{eq:lcs-instanton-intro}.

We will first carry out the asymptotic study of $\lcs$ instantons near the punctures.
For this study of asymptotic convergence result at the punctures and the relevant index
theory, it turns out to be useful to regard \eqref{eq:lcs-instanton-intro} as a version of
gauged sigma model with abelian Hick's field. It is also important to employ the notion of
asymptotic contact instanton at each puncture,
which is a massless instanton on $\R\times S^1$
canonically associated to any finite energy contact instantons.
It also gives rise to an asymptotic Hick's field, which is a holomorphic one-form
that appears as the asymptotic limit of the complex-valued $(1,0)$-form
$$
\chi= f^*d\theta + \sqrt{-1} w^*\lambda.
$$
The following asymptotic invariant is useful to introduce in relation to the
precise study of asymptotic behavior of contact instantons near punctures.

\begin{defn}[Asymptotic Hick's charge] Let $(\Sigma, j)$ be a closed Riemann surface
and $\dot \Sigma$ its associated punctured Riemann surface. Let $p$ be a given puncture of $\dot \Sigma$.
We define the \emph{asymptotic Hick's charge} of the instanton
$w: \dot \Sigma \to Q$ with finite energy with bounded gradient to be the complex number
$$
Q(p) + \sqrt{-1} T(p)
$$
defined by
\bea\label{eq:asymp-charge-intro}
Q(p) & = & \int_{\del_{\infty;p}\overline \Sigma}w^*\lambda\circ j = -\int_{S^1} \Re \chi(0,t)\, dt \\
T(p) & = & \int_{\del_{\infty;p}\overline \Sigma}w^*\lambda = \int_{S^1} \Im \chi(0,t)\, dt
\eea
where $z = e^{-2\pi(\tau + it)}$ are the analytic coordinates of $D_r(p)$ centered at $p$, and
$\overline \Sigma$ is the real blow-up of $\dot \Sigma$ performed at all the given puncture.
We call $Q(p)$ the \emph{contact instanton charge} of $w$ at $p$
and $T(p)$ the \emph{contact instanton action} of $w$ at $p$.
\end{defn}
We define the asymptotic Hick's field (or charge) of a map $w: \C \to Q$ at infinity
by regarding $\infty$ as a puncture associated to $\C \cong \C P^1 \setminus \{\infty\}$.

We will prove a removable singularity result (see Theorem \ref{thm:c=0}) that
if $Q(p) = 0 = T(p)$ $w$ is smooth across $p$ and so the
puncture $p$ is removable \emph{under a suitable finite energy condition} whose description is
one of the important themes of the present paper.

\begin{thm}[Compare with Proposition 8.3 \cite{oh:contacton}]\label{thm:pole-structure-intro}
Let $\eta \in H^1(\dot \Sigma,\Z)$ be a given charge class on a punctured Riemann surface
$\dot \Sigma = \Sigma = \{p_1, \cdots, p_k\}$.
Let $u = (w,f)$ be an $\lcs$ instanton on $\dot \Sigma$ in class $\eta$.
Consider the complex-valued one-form on $\dot \Sigma$ defined by
\be\label{eq:chi-intro}
\chi: = f^*d\theta + \sqrt{-1} w^*\lambda.
\ee
Let $p \in \{p_1, \cdots, p_k\}$ and let $z$ be an analytic coordinate at $p$. Suppose
$$
E(u) < \infty.
$$
Then for any given sequence $\delta_j \to 0$ there exists a subsequence, still denoted by $\delta_j$, and a conformal diffeomorphism
$\varphi_j: [-\frac{1}{\delta_j}, \infty) \times S^1 \to D_{\delta_j}(p)\setminus \{p\}$ such that
the one form $\varphi_j^*\chi$ converges to a bounded holomorphic one-form
$\chi_\infty$ on $(-\infty, \infty) \times S^1$.
\end{thm}
We would like to emphasize that at the moment, the limiting holomorphic one-form $\chi_\infty$
may depend on the choice of subsequence.

\subsection{Charge class, $\lcs$ instanton energy and $\epsilon$-regularity}

Now we consider the general case where $Q(p) \neq 0$ for some $p$ in general.
As mentioned above, this is the most troublesome case for the study of contact
instantons in \cite{oh-wang:CR-map1, oh-wang:CR-map2,oh:contacton}.
It turns out that in the current $\lcs$ context, there is a nice way of treating
this non-zero charge case by further decomposing the moduli space
into the union of suitable sub-moduli spaces classified by some homotopy class
of maps for the domain curves whose description is now in order.

We denote by $\dot \Sigma$ the punctured Riemann surfaces of a closed Riemann
surface $\Sigma$. We can write
$$
f^*d\theta = \beta_\eta + d\widetilde f
$$
for some \emph{harmonic one-form} and a \emph{real-valued} function
$\widetilde f: \dot \Sigma \to \R$ under a suitable asymptotic property on $f^*d\theta$.
(We refer to Subsection \ref{subsec:energy} for the description of the relevant
asymptotic condition.)
With respect to the K\"ahler metric which is strip-like near each puncture, the map
$$
(w,\widetilde f): [0, \infty) \times S^1 \to Q \times \R
$$
satisfies the following \emph{perturbed} pseudoholomorphic curve equation
$$
\delbar^\pi w = 0 , \, w^*\lambda \circ j - d\widetilde f = \beta_\eta
$$
on $\dot \Sigma$. Then one can employ Hofer-type $\lambda$-energy using the pair $(w,\widetilde f)$.
(We refer readers to Section \ref{sec:lcs-instanton} for the details on how its
definition goes.) With this energy, we prove the following $\epsilon$-regularity result.
(Compare this with a similar result for the contact instantons in \cite[Theorem 7.4]{oh:contacton}.)

In the current setting of $\lcs$ instanton map, it is not obvious what would be
the precise form of relevant $\epsilon$-regularity statement. We formulate this $\epsilon$-regularity
theorem in terms of the associated contact instantons.

The following is by now a standard definition in contact topology.

\begin{defn} Let $\lambda$ be a contact form of contact manifold $(Q,\xi)$.
Denote by $\frak R eeb(Q,\lambda)$ the set of closed Reeb orbits.
We define $\operatorname{Spec}(Q,\lambda)$ to be the set
$$
\operatorname{Spec}(Q,\lambda) = \left\{\int_\gamma \lambda \mid \lambda \in \frak Reeb(Q,\lambda)\right\}
$$
and call the \emph{action spectrum} of $(Q,\lambda)$. We denote
$$
T_\lambda: = \inf\left\{\int_\gamma \lambda \mid \lambda \in \frak Reeb(Q,\lambda)\right\}.
$$
\end{defn}
We set $T_\lambda = \infty$ if there is no closed Reeb orbit.
This constant $T_\lambda$ will enter in a crucial way in the following
period gap theorem of $\lcs$ instantons. (See \cite{hofer} for
the first such appearance in the study of pseudoholomorphic curves in the symplectization
$Q \times \R$.)

\begin{thm}\label{thm:e-regularity-intro}
Denote by $D^2(1)$ the closed unit disc and let $u = (w,f)$ be an $\text{\rm lcs}$ instanton
defined on $D^2(1)$ so that $w:D^2(1)\to Q$ satisfies
$$
\delbar^\pi w = 0, \, w^*\lambda\circ j = f^*d\theta.
$$
Assume the vertical energy bound $E^\perp (w) < K_0$ defined in Definition \ref{defn:vertical-energy}.
Then for any given $0 <\epsilon < T_\lambda$ and $w$ satisfying
$E^\pi(w) < T_\lambda - \epsilon$, and for a smaller disc $D' \subset \overline D' \subset D$,
there exists some $K_1 = K_1(D', \epsilon,K_0) > 0$
\be\label{eq:dwC0-intro}
\|dw\|_{C^0;D'} \leq K_1
\ee
where $K_1$ depends only on $(Q,\lambda,J)$, $\epsilon$, $D' \subset D$.
\end{thm}

\subsection{Linearization and the Fredholm theory}

The fixation of a charge class for the $\lcs$ instanton also enables us to develop
the Fredholm theory and to construct a compactification, because it rules out the
phenomenon of \emph{appearance of spiraling instantons along the Reeb core} in the
asymptotic subsequence limit at the puncture. (See \cite[Section 6]{oh-wang:CR-map1} for
the detailed explanation on the appearance of this phenomenon for the case of
contact instantons.)

We consider the map
$$
\Upsilon(w,f) = \left(\delbar^\pi w, w^*\lambda \circ j - f^*d\theta \right)
$$
whose zero set is the set of $\lcs$ instantons by definition.
For the optimal expression of the linearization map and its relevant
calculations, we use the $\mathfrak{lcs}$-fication connection $\nabla$ of $(Q \times S^1,\lambda,J)$
which is the lcs-lifting of the contact triad connection introduced in \cite{oh-wang:CR-map1}.
(See Section \ref{sec:connection} for the details.)
We refer readers to \cite{oh-wang:CR-map1}, \cite{oh:contacton} for the unexplained notations.

\begin{thm}\label{thm:linearization-intro} Let $u = (w,f): \dot \Sigma Q \times S^1$
be an $\lcs$ instanton of a given charge class $[u]  =\eta$.
We decompose $d\pi = d^\pi w + w^*\lambda\otimes R_\lambda$
and $Y = Y^\pi + \lambda(Y) R_\lambda$, and $X = (Y, v) \in \Omega^0(u^*T(Q \times S^1))$.
Denote $\kappa = \lambda(Y)$ and $\upsilon = d\theta(v)$. Then we have $D\Upsilon(u) = D\Upsilon_1(u) + D\Upsilon_2(u)$
with
\bea
D\Upsilon_1(u)(Y,v) & = & \delbar^{\nabla^\pi}Y^\pi + B^{(0,1)}(Y^\pi) +  T^{\pi,(0,1)}_{dw}(Y^\pi) \nonumber\\
&{}& \quad + \frac{1}{2}\kappa \cdot  \left((\CL_{R_\lambda}J)J(\del^\pi w)\right)
\label{eq:Dwdelbarpi-intro}\\
D\Upsilon_2(u)(Y,v) & = &  (\CL_Y \lambda) \circ j- \CL_v d\theta = d\kappa \circ j - d\upsilon
+ Y \rfloor d\lambda \circ j
\nonumber\\
\label{eq:Dwddot-intro}
\eea
where $B^{(0,1)}$ and $T_{dw}^{\pi,(0,1)}$ are the $(0,1)$-components of $B$ and
$T_{dw}^{\pi,(0,1)}$ respectively, where $B, \, T_{dw}^\pi: \Omega^0(w^*TQ) \to \Omega^1(w^*\xi)$ are
the zero-order differential operators given by
$$
B(Y) =
- \frac{1}{2}  w^*\lambda \left((\CL_{R_\lambda}J)J Y\right)
$$
and
$$
T_{dw}^\pi(Y) = \pi T(Y,dw).
$$
\end{thm}
More succinctly, we can express the operator $D\Upsilon(u)$ in a matrix form
\be\label{eq:matrix-form-intro}
\left(
\begin{matrix}
 \delbar^{\nabla^\pi} + B^{(0,1)} +  T^{\pi,(0,1)}_{dw} &, & \frac{1}{2}(\cdot) \cdot
  \left((\CL_{R_\lambda}J)J(\del^\pi w)\right)\\
\left((\cdot)^\pi \rfloor d\lambda\right)\circ j &, & \delbar
\end{matrix}
\right).
\ee
\begin{rem}
We would like to highlight that our linearization formula is coordinate-free and
is written in terms of the \emph{$\mathfrak{lcs}$-fication of contact triad connection from \cite{oh-wang:connection}}.
(See Section \ref{sec:connection}, especially Remark \ref{rem:connection}
for some more details on this connection.) The same formula
equally applies to the pseudoholomorphic curves in the symplectization as a special case.
A novelty of our coordinate-free formula is that  it equips each individual term in the formula
of the linearization operator with natural tensorial geometric meaning in terms of the
contact triad $(Q,\lambda, J)$ and its triad connection. Compare this with the coordinate-dependent formula
for the linearization operator appearing in the literature such as \cite{bourgeois}, \cite{behwz}
on the Fredholm analysis of the moduli space of pseudoholomorphic curves in the symplectization.
\end{rem}

Then noting that the off-diagonal terms of \eqref{eq:matrix-form-intro} are zero-order operators,
by the continuous invariance of the Fredholm index, we obtain
\be\label{eq:indexDXiw}
\operatorname{Index} D\Upsilon_{(\lambda,T)}(w) =
\operatorname{Index} \left(\delbar^{\nabla^\pi} + T^{\pi,(0,1)}_{dw}  + B^{(0,1)}\right)
 + \operatorname{Index}(\delbar).
\ee
Therefore it remains to compute the latter two indices which follows from by now
standard Riemann-Roch type index formulae
(See and compare with \cite[p.52]{bourgeois} for a relevant formula.)

\begin{thm}\label{thm:indexforDUpsilon-intro} We fix a trivialization
$\Phi: E \to \overline \Sigma$ and denote
by $\Psi_i^+$ (resp. $\Psi_j^-$) the induced symplectic paths associated to the trivializations
$\Phi_i^+$ (resp. $\Phi_j^-$) along the Reeb orbits $\gamma^+_i$ (resp. $\gamma^-_j$) at the punctures
$p_i$ (resp. $q_j$) respectively. Then we have
\bea
&{}&
\operatorname{Index} D\Upsilon_{(\lambda,T)}(u) \\
& = & n(2-2g-s^+ - s^-) + 2c_1(w^*\xi)\nonumber\\
&{}&  + \sum_{i=1}^{s^+} \mu_{CZ}(\Psi^+_i)
- \sum_{j=1}^{s^-} \mu_{CZ}(\Psi^-_j)\nonumber \\
&{}&  +
\sum_{i=1}^{s^+} (2m(\gamma^+_i)+1) + \sum_{j=1}^{s^-}( 2m(\gamma^-_j)+1)  - 2g.
\eea
\end{thm}

We refer to Section \ref{sec:Fredholm} for the details of the  Fredholm theory,
and postpone elsewhere for the construction of compactified moduli space and its applications.

\medskip

{\bf Acknowledgement:}  We would like to thank the unknown referee for her/his careful reading of
the paper and pointing out many careless typos and incorrect English expressions which we appreciate
very much.

\bigskip

\noindent{\bf Convention:}

\medskip

\begin{itemize}
\item The Hamiltonian vector field on symplectic manifold $(P, \omega)$ is defined by $X_H \rfloor \omega = dH$.
\item {(Contact Hamiltonian)} The contact Hamiltonian of a time-dependent contact vector field $X_t$ is
given by
$$
H: = - \lambda(X_t).
$$
We denote by $X_H$ the contact vector field whose associated contact Hamiltonian is given by $H = H(t,x)$.\item For given Lee form $\mathfrak b$, we define the operator $d^{\mathfrak b}$ action on $\Omega^*(M)$ by
\be\label{eq:db}
d^{\mathfrak b} \alpha = d\alpha + {\mathfrak b} \wedge \alpha.
\ee
\end{itemize}
These convention are consistent with that of \cite{oh:book}, \cite{oh:contacton-Legendrian-bdy} and \cite{le-oh:lcs}
respectively.

\section{Banyaga $\lcs$ manifolds and contact mapping tori}

In this section, we briefly summarize basic geometric properties of
the $\mathfrak{lcs}$-fication of a contact manifold $(Q,\lambda)$ and
of the pseudoholomorphic curves thereon that we are going to study.

\subsection{$\mathfrak{lcs}$-fication of contact manifolds}

As the starting point towards the Floer theory on general $\lcs$ manifolds, we
consider a special class of $\lcs$ manifolds, which we call \emph{Banyaga $\lcs$ manifolds}.
\begin{defn}[Banyana $\lcs$ structure \cite{banyaga:lcs}]\label{defn:banyaga-lcs}
Let $(Q,\lambda)$ be any contact manifold. A Banyaga $\lcs$ form on $Q \times S^1$ is defined to
be
$$
d^{\frak b}\lambda = d \lambda + {\frak b} \wedge \lambda =: \omega _{\lambda}, \quad \frak b = \pi^*d\theta.
$$
\end{defn}
From now on we will just denote by $d\theta$ $\pi^*d\theta$ slightly abusing notation as long as
there is no danger of confusion.

The main purpose of the present paper is to develop a Floer theory on this class of $\lcs$
manifolds.

Let $\CV: = \ker d\lambda$ i.e.,
$$
\mathcal{V}_p = \{ v \in T _{p} M  \mid d\lambda(v, \cdot)=0 \}
$$
at each $p \in M$. Then it follows that
$\mathcal{V}$ is a 2-dimensional distribution: $\CV_p $
has dimension at least 2 since $d \lambda$ cannot be symplectic since $M $ is closed,
and has dimension at most 2 since
$\omega_\lambda = d \lambda + d\theta \wedge \lambda$ is non-degenerate.

Next we let $\widetilde \xi$ denote the co-vanishing distribution that is $\widetilde \xi _{p} $
is the $\omega_\lambda$-orthogonal complement to $\CV_p  $. With these definitions, we
have a canonical splitting
\be\label{eq:splitting}
TM = \widetilde \xi \oplus \CV.
\ee
It follows that
$$
d\pi_Q(\widetilde \xi) = \xi, \, d\pi_{S^1}(\widetilde \xi)  = 0
$$
i.e., $\widetilde \xi  = \xi \oplus 0$ in terms of the splitting
$TM = (\xi \oplus \span_\R\{R_\lambda\}) \oplus TS^1 $.
Here $\pi_Q: M \to Q$ is the natural projection and
$\xi = \xi_\lambda \subset TQ$ is the contact distribution of $(Q,\lambda)$.

\begin{rem} One may try to
generalize the above discussion to a
more general Lichnerowitz exact $\lcs$ structure $\omega$ on $M =Q^{2n-1} \times S^1$,
with $Q$ a general manifold equipped with a general 1-form $\lambda$ on $M$ such that the 2-form
$\omega = d \lambda + d\theta \wedge \lambda$ is non-degenerate.
See Subsection \ref{subsec:mapping-tori} below.
We will however focus on the above Banyaga $\lcs$ manifolds leaving the general
mapping tori case (or even the general case of lcs manifolds) for a future work so that
we can directly utilize the analysis for
the \emph{contact instantons} developed in \cite{oh-wang:CR-map1} by Wang and
the first named author.
\end{rem}

Now denote a map $u: \dot \Sigma \to M$ as $u = (w,f)$ where $w: \dot \Sigma \to Q$ and
$f: \dot \Sigma \to S^1$ are the components of $u$ for $Q$ and $S^1$ respectively.
We have the natural decomposition of $du$ into
$$
du =  dw \oplus df
$$
induced by the product structure $M = Q \times S^1$. Under the presence of the contact form $\lambda$,
$TQ = \xi \oplus \span_\R\{R_\lambda\}$ which in turn induces the splitting
$$
du = \pi_\lambda \circ dw  \oplus  w^*\lambda \otimes R_\lambda \oplus df.
$$
We denote $d^\pi w = \pi_\lambda \circ dw$ following \cite{oh-wang:CR-map1}.

On the other hand, the splitting \eqref{eq:splitting} induces another splitting
$$
du = \pi_{\widetilde \xi}\circ du \oplus \pi_\CV\circ du.
$$
By definition, we have
$$
\pi_\xi(\pi_{\widetilde \xi} \circ du) =  d^\pi w
$$
and
$$
\pi_{\widetilde \xi}^\perp(\pi_\CV\circ du) = w^*\lambda \otimes R_\lambda \oplus df
$$
where $\pi_{\widetilde \xi}^\perp$ is the projection to
$(\span_\R\{R_\lambda\} \oplus TS^1)$ with respect to the splitting
$$
TM =\xi  \oplus \span_\R\{R_\lambda\} \oplus TS^1.
$$
We will denote by
$\Pi_\xi, \, \Pi_\xi^\perp: TM \to TM$ the associated idempotents.

\subsection{Mapping tori of contactomorphisms}
\label{subsec:mapping-tori}

Another class of natural lcs manifolds arises as the mapping cylinder of contactomorphisms
whose description is now on order. This is a generalization of
Banyaga's lcs manifold on $Q \times S^1$ to the general mapping tori associated
to any contact diffeomorphism $\phi$. Banyaga's lcs manifold corresponds to
the mapping torus of the identity map.

Let $(Q,\xi)$ be a contact manifold and $\phi \in \Cont(Q,\xi)$ ba a
contactomorphism. By definition, we have
$$
d\phi(\xi) \subset \xi.
$$
We consider the product $Q \times \R$ and the distribution given by
$$
\xi \oplus \R \langle \frac{\del}{\del s} \rangle \subset T(Q \times \R).
$$
Since $d\phi(\xi) \subset \xi$, the distribution descends to
the mapping torus
$$
M_\phi: = \frac{Q \times \R}{(t,\phi(x)) \sim (t+1,x)}.
$$
We denote by $\widetilde \xi$ the resulting distribution
$$
\widetilde \xi = \left[\xi  \oplus \R \langle \frac{\del}{\del s} \rangle\right]
\subset TM_\phi
$$
on $M_\phi$.

\begin{rem} We would like to emphasize that the above mapping torus is
well-defined for the contact manifold $(M,\xi)$ not for the one $(M,\lambda)$
with a given contact form: Unless $\phi$ is strict, i.e., $\phi^*\lambda = \lambda$,
the obvious pull-back form $\pi^*\lambda$ on $Q \times \R$ does not descend but
the contact distribution does.
\end{rem}

Since $\widetilde \xi$ is cooriented if $\xi$ is, we can take a one-form $\kappa$
on $M_\phi$ such that
\be\label{eq:kappa}
\ker \kappa = \widetilde \xi.
\ee
\begin{prop} For each contactomorphism $\phi \in \Cont(Q,\xi)$, the
two form $d \kappa + d\theta \wedge \kappa =: \omega_\kappa$ is an lcs form.
Furthermore if $\kappa'$ is another such form satisfying \eqref{eq:kappa},
the two lcs forms $\omega_\kappa$ and $\omega_{\kappa'}$ are (positively) conformally
diffeomorphic.
\end{prop}
\begin{proof} We take the differential
$$
d\omega_\kappa = d(d \kappa + d\theta \wedge \kappa) = - d\theta \wedge d \kappa
$$
which is equivalent to
$$
0 = d\omega_\kappa + d\theta \wedge d \kappa =  d\omega_\kappa + d\theta \wedge \omega_\kappa = d^{\frak b}\omega_\kappa
$$
with $\frak b = d\theta$.

Now we prove non-degeneracy. Consider the exact sequence
$$
0 \to \ker \kappa \to TM_\phi \to TM_\phi/\ker \kappa \to 0
$$
and take the splitting of the sequence
$$
TM_\phi = \ker \kappa \oplus \R\langle R \rangle.
$$
We take the $(n+1)$-th power
$$
(\omega_\kappa)^{n+1} = (d\kappa)^{n+1} + (d\kappa)^n \wedge d\theta \wedge \kappa.
$$
Since $\dim \ker \kappa = 2n$, $(d\kappa)^{n+1} = 0$.
On the other hand, the second term is nowhere vanishing because we have
$$
(d\kappa)^n \wedge d\theta \wedge \kappa\left(e_1,f_1,e_2,f_2,\ldots, \del_\theta, R\right) = 1
$$
for any Darboux basis $\{e_1,f_1,e_2,f_2,\ldots,e_n,f_n\}$ and a choice of $R \in TQ$ such that
$$
\lambda(R) = 1
$$
for $R \in T_yQ$. Therefore $(\omega_\kappa)^{n+1}$ is nowhere vanishing and hence $\omega_\kappa$ is
non-degenerate.

Now for the last part of the proposition, suppose that $\kappa _{0}, \kappa _{1}  $ are a pair of $1$-forms as above.
We will find a pair $(\phi,f)$ of  a diffeomorphism $\phi: M \to M$ and a function $f: M \to \R$
satisfying \eqref{eq:+lcs-definition} with $\frak b = \frak b'$ with $\frak b = d\theta$.
   Let $\{\kappa _{t} = t\kappa_1 +(1-t) \kappa_0 \}$, $t \in [0,1]$, be the convex linear combination of $\kappa _{i} $, so in particular $\ker \kappa _{t} = \widetilde{\xi} $.  This gives a homotopy
   $$
   \{\omega _{t} := \omega _{\kappa _{t} } =  d ^{\frak b} \kappa _t  \}
   $$
   of $\lcs$ forms, with the property that
   $$
   \frac{d}{d\tau}\Big| _{\tau=t}  \omega _{\tau} = d ^{\frak b} \rho _{t}
   $$
   for a family of smooth 1-forms $\{\rho _{t} \} $,
   $
   \rho_t = \frac{d}{d\tau}| _{\tau=t} \kappa_t (= \kappa_1 - \kappa_0).
   $
At this point we may apply the version of lcs Moser's argument as it appears in \cite[Theorem 4]{banyaga:lcs}.
(Alternatively, we may also refer readers to the proof of \cite[Theorem 4.2]{le-oh:lcs} for more concrete details
in a similar context.)
More specifically, Banyaga proves that under the current circumstance there exists a family of the pairs $(\phi_t,f_t)$
that satisfy
\be\label{eq:+lcs-morphism}
\phi_t^*\omega_t= e^{f_t} \omega_{\kappa_0}, \, \phi_t^* d\theta = d\theta - df_t
\ee
with $\phi_0 = id, \, f_0 = 0$.
\end{proof}

\begin{defn}[$\text{\rm lcs}$ mapping torus]\label{defn:lcs-mapping-torus} Let $\phi \in \Cont(Q,\xi)$. We call the pair $(M_\phi,\omega_\kappa)$
the $\text{\rm lcs}$ mapping torus of the contactomorphism $\phi$.
\end{defn}

\begin{rem}
Analysis of lcs instantons developed in the present paper on the $\mathfrak{lcs}$-fication $Q \times S^1$ of
contact manifold $Q$, which corresponds to the mapping torus $\phi = id$, can be promoted to
one on the contact mapping tori. From the point of view of study of contact dynamics
via perturbed contact instantons in \cite{oh:contacton-Legendrian-bdy},
considering the mapping tori corresponds to
discretizing the contact Hamiltonian dynamics. In this way, we can rule out
the troublesome phenomenon of the appearance of `spiraling instantons along the Reeb core' by fixing the
charge class. We hope to investigate the study of lcs-instantons on
the contact mapping tori and its application to the study of
contactomorphisms and their loops elsewhere.
\end{rem}

\section{Contact triad connection and its $\lcs$-lifting}
\label{sec:connection}

Let$(Q, \lambda, J)$ be a contact triad of dimension $2n+1$ for the contact manifold $(M, \xi)$,
and equip with it the contact triad metric $g=g_\xi+\lambda\otimes\lambda$.
In \cite{oh-wang:connection}, the authors introduced the \emph{contact triad connection} associated to every contact triad $(Q, \lambda, J)$ with the contact triad metric and proved its existence and uniqueness.

\begin{thm}[Contact Triad Connection \cite{oh-wang:connection}]\label{thm:connection}
For every contact triad $(Q,\lambda,J)$, there exists a unique affine connection $\nabla$, called the contact triad connection,
 satisfying the following properties:
\begin{enumerate}
\item The connection $\nabla$ is  metric with respect to the contact triad metric, i.e., $\nabla g=0$;
\item The torsion tensor $T$ of $\nabla$ satisfies $T(R_\lambda, \cdot)=0$;
\item The covariant derivatives satisfy $\nabla_{R_\lambda} R_\lambda = 0$, and $\nabla_Y R_\lambda\in \xi$ for any $Y\in \xi$;
\item The projection $\nabla^\pi := \pi \nabla|_\xi$ defines a Hermitian connection of the vector bundle
$\xi \to Q$ with Hermitian structure $(d\lambda|_\xi, J)$;
\item The $\xi$-projection of the torsion $T$, denoted by $T^\pi: = \pi T$ satisfies the following property:
\be\label{eq:TJYYxi}
T^\pi(JY,Y) = 0
\ee
for all $Y$ tangent to $\xi$;
\item For $Y\in \xi$, we have the following
$$
\del^\nabla_Y R_\lambda:= \frac12(\nabla_Y R_\lambda- J\nabla_{JY} R_\lambda)=0.
$$
\end{enumerate}
We call $\nabla$ the contact triad connection.
\end{thm}
From this theorem, we see that the contact triad connection $\nabla$ canonically induces
a Hermitian connection $\nabla^\pi$ for the Hermitian vector bundle $(\xi, J, g_\xi)$, and we call it the \emph{contact Hermitian connection}.

Moreover, the following fundamental properties of the contact triad connection was
proved in \cite{oh-wang:connection}, which will be used to perform tensorial calculations later.

\begin{cor}\label{cor:connection}
Let $\nabla$ be the contact triad connection. Then
\begin{enumerate}
\item For any vector field $Y$ on $Q$,
\be\label{eq:nablaYX}
\nabla_Y R_\lambda = \frac{1}{2}(\CL_{R_\lambda}J)JY;
\ee
\item $\lambda(T|_\xi)=d\lambda$.
\end{enumerate}
\end{cor}
We refer readers to \cite{oh-wang:connection} for more discussion on the contact triad connection
and its relation with other related canonical type connections.

Now we define the liftings of contact triad and of $(Q,\lambda)$ to
its $\mathfrak{lcs}$-fication $(Q \times S^1,\omega_\lambda)$. Recalling
$T(Q \times S^1) = TQ \oplus TS^1$, we define a connection on $T(Q \times S^1)$ as follows.

\begin{defn}[$\mathfrak{lcs}$-fication of contact triad connection] Let $(Q,\lambda,J)$ be a contact
triad and $\nabla$ the associated contact triad connection given in Theorem \ref{thm:connection}.
Denote by $\overline J$ the unique $\lambda$-admissible almost complex structure
whose restriction to $TQ$ is $J$.
We call the connection $\overline \nabla$ on $T(Q \times S^1)$ be the \emph{$\mathfrak{lcs}$-fication connection}
for $(Q,\lambda, J)$ if it satisfies the following:
\begin{enumerate}
\item[(7)] it satisfies all the properties of (1) - (6) above,
\item[(8)] $\overline \nabla_{\frac{\del}{\del \theta}}\frac{\del}{\del \theta}= 0$ and
$\overline \nabla_Y \frac{\del}{\del \theta} \in \xi$ for all $Y \in \xi$,
\item[(9)] $\del_Y^{\overline \nabla}\frac{\del}{\del \theta} = 0$ for all $Y \in \xi$,
\item[(10)] $\overline \nabla_{R_\lambda} \frac{\del}{\del \theta} = \overline \nabla_{\frac{\del}{\del \theta}} R_\lambda = 0$,
\item[(11)] $\overline \nabla$ is $\overline J$-linear.
\end{enumerate}
\end{defn}
It is immediate to check that the above requirement indeed defines a connection on $T(Q \times S^1)$.
Obviously $(10)$ implies the torsion $\overline T$ of $\overline \nabla$ satisfies
$\overline T\left(R_\lambda, \frac{\del}{\del \theta}\right) = 0$.

Furthermore we also have

\begin{prop} The $\mathfrak{lcs}$-fication connection $\overline \nabla$ on $T(Q \times S^1)$ satisfies
$$
\overline \nabla_Y \frac{\del}{\del \theta} = - \frac12 (\CL_{R_\lambda} J) Y \in \xi
$$
for all $Y \in \xi$.
\end{prop}
\begin{proof} By the definition of $\lambda$-admissible almost complex structure in
Definition \ref{defn:lambda-admissible-J}, we have
$\overline J\frac{\del}{\del \theta} = R_\lambda$. Therefore using the $\overline J$-linearity of the connection
and compatibility with the contact triad connection $\nabla$ in addition,
and applying the relation $J \CL_YJ = - (\CL_YJ) J$ a couple of times, we derive
\beastar
\overline \nabla_Y \frac{\del}{\del \theta} & = & - \overline \nabla_Y \overline J R_\lambda
= -\overline J \overline \nabla_Y R_\lambda \\
& = &  - \overline J \frac12 \left(\CL_{R_\lambda} J\right) J R_\lambda =
- \frac12 \CL_{R_\lambda}J R_\lambda
\eeastar
which finishes the proof.
\end{proof}

Usage of this connection is not essential for the main study of pseudoholomorphic
curves but will simplify geometric calculations and many formulae that appear in our elliptic estimates
\cite{oh-wang:CR-map1}
and in our expression of the linearization operator (see Section \ref{subsec:linearization}). This enables us to provide transparent geometric
interpretation of the various terms appearing in the linearization operator
as in \cite{oh-wang:CR-map1,oh-wang:CR-map2}.

\begin{rem}\label{rem:connection}
\begin{enumerate}
\item One can promote the above definition of $\mathfrak{lcs}$-fication of contact triad connection to
the $\mathfrak{lcs}$-fication of the fiberwise contact triad connection to the contact mapping tori whose
detailed study is postponed to elsewhere.
\item
The $\mathfrak{lcs}$-fication connection $\overline \nabla$ even for the symplectization
is not the canonical connection of the \emph{symplectization} as an almost K\"ahler manifold
\cite{gauduchon}, \cite{kobayashi}:
this is manifested by the nonvanishing $\overline \nabla R_\lambda \neq 0$ while
all other metric connections
on contact manifolds in the literature require $R_\lambda$ to be a Killing vector field and
so also $\overline \nabla R_\lambda = 0$ in their symplectizations.
(See \cite[Introduction]{oh-wang:connection} for the relevant discussion on the relationship between
the contact triad connection and other connections used in the literature on the natural
connections on contact manifolds.)
\end{enumerate}
\end{rem}

\emph{From now on, by an abuse of notation, we will omit the overline from the notation and just denote by $\nabla$ this
$\mathfrak{lcs}$-fication connection on $Q \times S^1$ of the contact triad connection $\nabla$ associated to $(Q,J,\lambda)$.}

\section{Review of contact Cauchy-Riemann maps and Hofer's energy}
\label{sec:CRmap}

\subsection{Contact Cauchy-Riemann maps}

In this section, we recall the basic definition and properties of the so called \emph{contact Cauchy
Riemann map} and \emph{contact instanton} introduced in \cite{oh-wang:CR-map1}.

We denote by $(\dot\Sigma, j)$ a punctured Riemann surface (including the case of closed Riemann surfaces without punctures).

\begin{defn}A smooth map $w:\dot\Sigma\to Q$ is called a \emph{contact Cauchy--Riemann map}
(with respect to the contact triad $(Q, \lambda, J)$), if $w$ satisfies the following Cauchy--Riemann equation
$$
\delbar_J^\pi w:=\delbar^{\pi}_{j,J}w:=\frac{1}{2}(\pi dw+J\pi dw\circ j)=0.
$$
\end{defn}

Recall that for a fixed smooth map $w:\dot\Sigma\to Q$,
the triple $(w^*\xi, w^*J, w^*g_\xi)$ becomes  a Hermitian vector bundle
over the punctured Riemann surface $\dot\Sigma$.  This  introduces a Hermitian bundle structure on
$Hom(T\dot\Sigma, w^*\xi)\cong T^*\dot\Sigma\otimes w^*\xi$ over $\dot\Sigma$,
with inner product given by
$$
\langle \alpha\otimes \zeta, \beta\otimes\eta \rangle =h(\alpha,\beta)g_\xi(\zeta, \eta),
$$
where $\alpha, \beta\in\Omega^1(\dot\Sigma)$, $\zeta, \eta\in \Gamma(w^*\xi)$,
and $h$ is the K\"ahler metric on the punctured Riemann surface $(\dot\Sigma, j)$.

Let $\nabla^\pi$ be the contact Hermitian connection.
Combining the pull-back of this connection and the Levi-Civita connection of the Riemann surface,
we get a Hermitian connection for the bundle $T^*\dot\Sigma\otimes w^*\xi \to \dot\Sigma$.
By a slight abuse of notation, we will still denote by $\nabla^\pi$ this combined connection.

The smooth map $w$ has an associated $\pi$-harmonic energy density
defined as the norm of the section $d^\pi w:=\pi dw$ of $T^*\dot\Sigma\otimes w^*\xi\to \dot\Sigma$.
In other words, it is the function $e^\pi(w):\dot\Sigma\to \R$ defined by
$
e^\pi(w)(z):=|d^\pi w|^2(z).
$
(Here we use $|\cdot|$ to denote the norm from $\langle\cdot, \cdot \rangle$ which should be clear from the context.)

Similarly to the case of pseudoholomorphic curves on almost K\"ahler manifolds,
we obtain the following basic identities, for whose proofs we refer readers to \cite{oh-wang:CR-map1}.

\begin{lem}[Lemma 3.2 \cite{oh-wang:CR-map1}]\label{lem:omega-area}
Fix a K\"ahler metric $h$ on $(\dot\Sigma,j)$,
and consider a smooth  map $w:\dot\Sigma \to Q$.  Then we have the following equations
\begin{enumerate}
\item $e^\pi(w):=|d^\pi w|^2 = |\del^\pi w| ^2 + |\delbar^\pi w|^2$;
\item $2\, w^*d\lambda = (-|\delbar^\pi w|^2 + |\del^\pi w|^2) \,dA $
where $dA$ is the area form of the metric $h$ on $\dot\Sigma$;
\item $w^*\lambda \wedge w^*\lambda \circ j = - |w^*\lambda|^2\, dA$.
\end{enumerate}
As a consequence, if $w$ satisfies $\delbar^\pi w=0$, then
\be\label{eq:onshell}
|d^\pi w|^2 = |\del^\pi w| ^2 \quad \text{and}\quad w^*d\lambda = \frac{1}2|d^\pi w|^2 \,dA.
\ee
\end{lem}

We call a map $w:\dot \Sigma \to Q$ a \emph{contact Cauchy-Riemann map} if $w$ satisfies
$\delbar^\pi w=0$.  The contact Cauchy--Riemann equation itself is \emph{not} an elliptic system
since the symbol is of rank $2n$ which is $1$ dimension lower than $TM$.
Here the closedness condition $d(w^*\lambda\circ j)=0$
leads to an elliptic system (see \cite{oh:contacton} for an explanation)

\subsection{Contact instantons}

The following definition is introduced in \cite{oh-wang:CR-map1,oh-wang:CR-map2} and
its analysis of the moduli space relevant to the equation has been developed
therein and in \cite{oh:contacton}.

\begin{defn}[Contact instantons \cite{oh-wang:CR-map1}] A contact Cauchy-Riemann map $w:(\dot \Sigma, j) \to (Q,J)$ is called
a \emph{contact instanton} if it satisfies $d(w^*\lambda \circ j) = 0$ in addition.
\end{defn}

We call the defining equation of a contact instanton
\be\label{eq:contacton}
\delbar^\pi w = 0, \quad d(w^*\lambda \circ j) = 0
\ee
a \emph{contact instanton equation}.

We recall the following local elliptic estimates for any contact instantons $w$
proved in \cite{oh-wang:CR-map1}.

\begin{thm}[Theorem 1.6 \cite{oh-wang:CR-map1}]\label{thm:local-W12}
Let $(\dot \Sigma, j)$ be a punctured Riemann surface with a possibly
empty set of punctures. Equip $\dot \Sigma$ with a metric which is cylindrical
near each puncture. Let $w: \dot\Sigma \to M$ be a contact instanton.
For any relatively compact domains $D_1$ and $D_2$ in
$\dot\Sigma$ such that $\overline{D_1}\subset D_2$, we have
$$
\|dw\|^2_{W^{1,2}(D_1)}\leq C_1 \|dw\|^2_{L^2(D_2)} + C_2 \|dw\|^4_{L^4(D_2)},
$$
where $C_1, \ C_2$ are some constants which
depend only on $D_1$, $D_2$ and $(M,\lambda, J)$.
\end{thm}

We also establish the following iterative local $W^{2+k,2}$-estimates on punctured surfaces $\dot \Sigma$
in terms of the $W^{\ell,p}$-norms with $\ell \leq k+1$. Combined with
Theorem \ref{thm:local-W12}, this theorem in turn provides
a priori local $W^{2+k,2}$-estimates in terms of (local) $L^2$, $L^4$ norms of $|d^\pi w|$,
and $|w^*\lambda|$.

\begin{thm}[Theorem 1.7 \cite{oh-wang:CR-map1}] \label{thm:Wk2}
Let $w$ be a contact instanton.
Then for any pair of domains $D_1 \subset D_2 \subset \dot \Sigma$ such that $\overline{D_1}\subset D_2$,
$$
\int_{D_1} |(\nabla)^{k+1}(dw)|^2 \leq \int_{D_2} \CJ_{k}(d^\pi w, w^*\lambda).
$$
Here $\CJ_k$ is a polynomial function of degree up to $2k+4$ with nonnegative coefficients  of the norms of the covariant derivatives
of $d^\pi w, \, w^*\lambda$ up to $0, \, \ldots, k$ with degree at most $2k + 4$
whose coefficients depending on $J$, $\lambda$ and $D_1, \, D_2$ but independent of $w$.
\end{thm}
One can also directly derive $C^{k,\alpha}$ estimates instead as in \cite{oh:contacton-Legendrian-bdy}.
The following classification result of closed contact instantons was proved
in \cite{oh-wang:CR-map1}. (See also \cite[Proposition 1.4]{abbas} where Abbas
made a similar statement as a part of \cite[Proposition 1.4]{abbas}.)

\begin{prop}[Proposition 3.4, \cite{oh-wang:CR-map1}]\label{prop:abbas} Assume $w:\Sigma\to M$ is a smooth contact instanton from a closed Riemann surface.
Then
\begin{enumerate}
\item If $g(\Sigma)=0$, $w$ is a constant map;
\item If $g(\Sigma)\geq 1$, $w$ is either a constant or the locus of its image
is a \emph{closed} Reeb orbit.
\end{enumerate}
In particular, any such instanton  satisfies $[w] = 0$ in $H_2(Q;\Z)$ and so
is massless (i.e., $E^\pi(w) = 0$).
\end{prop}

Next we recall the asymptotic  behavior of contact instantons in general
from \cite{oh-wang:CR-map1}.

\begin{hypo}\label{hypo:basic}
Let $h$ be the metric on $\dot \Sigma$ given above.
Assume $w:\dot\Sigma\to Q$ satisfies the contact instanton equations \eqref{eq:contacton},
and
\begin{enumerate}
\item $E^\pi(w) <\infty$ (finite $\pi$-energy);
\item $\|d w\|_{C^0(\dot\Sigma)} <\infty$.
\end{enumerate}
\end{hypo}

Let $w$ satisfy Hypothesis \ref{hypo:basic}. We can associate two
natural asymptotic invariants at each puncture defined as
\bea
T & := & \frac{1}{2}\int_{[0,\infty) \times S^1} |d^\pi w|^2 \, dA + \int_{\{0\}\times S^1}(w|_{\{0\}\times S^1})^*\lambda\label{eq:TQ-T}\\
Q & : = & \int_{\{0\}\times S^1}((w|_{\{0\}\times S^1})^*\lambda\circ j).\label{eq:TQ-Q}
\eea
(Here we only look at positive punctures. The case of negative punctures is similar.)

\begin{rem}\label{rem:TQ}
For any contact instanton $w$, since $\frac{1}{2}|d^\pi w|^2\, dA=d(w^*\lambda)$, by Stokes' formula,
$$
T = \frac{1}{2}\int_{[s,\infty) \times S^1} |d^\pi w|^2\, dA + \int_{\{s\}\times S^1}(w|_{\{s\}\times S^1})^*\lambda, \quad
\text{for any } s\geq 0.
$$

Moreover, since $d(w^*\lambda\circ j)=0$, the integral
$$
\int_{\{s \}\times S^1}(w|_{\{s \}\times S^1})^*\lambda\circ j, \quad
\text{for any } s \geq 0
$$
does not depend on $s$ whose common value is nothing but $Q$.
\end{rem}
We call $T$ the \emph{asymptotic contact action}
and $Q$ the \emph{asymptotic contact charge} of the contact instanton $w$ at the given puncture.

The following theorem slightly strengthens the convergence results from
\cite{oh-wang:CR-map1}.

\begin{thm}[Theorem 6.4 \cite{oh-wang:CR-map1}]\label{thm:subsequence}
Let $\Sigma$ be a closed Riemann surface of genus 0 with
a finite number of marked points $\{p_1, \cdots, p_k\}$ for $k \geq 3$, and let $\dot \Sigma = \Sigma \setminus
\{p_1,\cdots, p_k\}$ be the associated punctured Riemann surface equipped with a metric as before.
Suppose that $w$ is a contact instanton map $w:\dot \Sigma
\to Q \times S^1 $ with finite total energy $E^\pi(w)$
and fix a puncture $p \in \{p_1, \cdots, p_k\}$.

Then for any given sequence $I=\{\tau_k\}$ with $\tau_k \to \infty$,
there exists a subsequence $I' \subset I$ and a closed parameterized Reeb orbit $\gamma = \gamma_{I'}$ of period $T$ and some $(\tau_0,t_0) \in \R \times S^1$   and a massless instanton $w_\infty(\tau,t)$
(i.e., $E^\pi(w_\infty) = 0$) on the cylinder $\R \times S^1$ such that
$$
\lim_{i \to \infty} w(\tau + \tau_{k_i},t) = w_\infty
$$
on $K \times S^1$ in the $C^\infty$-sense for every compact subset $K \subset \R$.

Furthermore $w_\infty$ has the formula
$w_\infty(\tau,t) = \gamma(-Q(p) \tau + T(p)\, t)$ where $\gamma$ is a Reeb trajectory, and for the case
of $Q = 0$ and $T\neq 0$, the trajectory is a closed Reeb orbit of $R_\lambda$ with period $T$.
\end{thm}

\begin{prop}[Corollary 6.5 \cite{oh-wang:CR-map1}] \label{cor:tangent-convergence}
Let $(w,f):[0, \infty)\times S^1\to Q$ satisfy the $\lcs$
instanton equations \eqref{eq:lcs-instanton} and Hypothesis \ref{hypo:basic}.
Then
\beastar
&&\lim_{s\to \infty}\left|\pi \frac{\del w}{\del\tau}(s+\tau, t)\right|=0, \quad
\lim_{s\to \infty}\left|\pi \frac{\del w}{\del t}(s+\tau, t)\right|=0\\
&&\lim_{s\to \infty}\lambda(\frac{\del w}{\del\tau})(s+\tau, t)= - Q, \quad
\lim_{s\to \infty}\lambda(\frac{\del w}{\del t})(s+\tau, t)=T
\eeastar
and
$$
\lim_{s\to \infty}|\nabla^l dw(s+\tau, t)|=0 \quad \text{for any}\quad l\geq 1.
$$
All the limits are uniform for $(\tau, t)$ in $K\times S^1$ with compact $K\subset \R$.
\end{prop}

If $\lambda$ is nondegenerate and $Q = 0, \, T \neq 0$, then the convergence $w(\tau,\cdot) \to \gamma(T\cdot)$
in Theorem \ref{thm:subsequence} is uniform.

\begin{rem} In a recent work \cite{oh:contacton-Legendrian-bdy}, the first named author
proved that the charge $Q$ always vanishes for the contact instantons with
Legendrian boundary conditions, i.e., for the maps satisfying
\be\label{eq:contacton-Legendrian-bdy}
\begin{cases}
\delbar^\pi w = 0, \, d(w^*\lambda \circ j) = 0 \\
w(\overline{z_iz_{i+1}}) \subset R_i, \quad i = 0, \ldots, k
\end{cases}
\ee
for a tuple of Legendrian submanifolds $(R_1,\cdots, R_k)$. It will be interesting
to develop the $\mathfrak{lcs}$-fication of this boundary value problem.
\end{rem}

\subsection{Canonical symplectization and Hofer's $\lambda$-energy; revisit}
\label{subsec:canonical}

In this subsection, we first recall the canonical symplectization of a contact manifold $(Q,\xi)$,
 which does not involve the choice of contact form.
We consider the subset
\be\label{eq:setalpha}
\{\alpha \in T^*Q \mid \alpha \neq 0, \, \ker \alpha = \xi \} \subset T^*Q \setminus \{0\}.
\ee
When $Q$ is oriented and a positive contact form $\lambda$ is given, we consider
the $(2n+2)$-dimensional submanifold $W$ of $T^*Q$
\be\label{eq:W}
W = \{\alpha \in T^*Q \setminus \{0\} \mid \ker \alpha = \xi, \, \alpha(\vec n) > 0\}
\ee
where $\vec n$ is a vector such that $\R \{\vec n\} \oplus \xi$ becomes a positively
oriented basis. Note that $W$ is a principal $\R_+$-bundle over $Q$ that is trivial.

We can lift a map $w: \dot \Sigma \to Q$ to a map $\widehat w:\dot \Sigma \to W$
if the contact manifold $(Q,\xi)$ is cooriented. In particular when the contact manifold
$(Q,\xi)$ is equipped with the contact form $\lambda$, $w$ has a canonical lift given by
\be\label{eq:canonical-lift}
\widehat w(z) : = \lambda(w(z)) \in T^*Q_{w(z)}.
\ee
We then examine the relationship between $w$ being a contact instanton and $\widehat w$
being a pseudoholomorphic curve on $W$ with respect to a scale-invariant almost complex
structure on $W$. We give a geometric description of Hofer's remarkable energy
introduced in \cite{hofer} in terms of this canonical symplectization. This energy
is the key ingredient needed in the bubbling-off analysis and so in the construction of the
compactification of the moduli spaces of pseudoholomorphic curves needed to develop the
symplectic field theory \cite{EGH}, \cite{behwz}. In this section,
we will then introduce its variant for the study of contact instanton maps whose
charge is not necessarily vanishing, i.e. $w^*\lambda \circ j$ does not have to be exact.

We now denote by
$$
i_W: W \hookrightarrow T^*Q
$$
the canonical embedding and by
$$
\pi_W: W \to Q
$$
the canonical projection which defines a principle $\R_+$-bundle. A generator of this $\R_+$ action is given by
the Euler vector field
$$
E(q,p): = \sum_{i=1}^{2n+1} p_i\frac{\del}{\del p_i}.
$$
We also denote
by $\Theta$ the Liouville
one-form on $T^*Q$. The basic proposition is that $W$ carries the canonical symplectic form
\be\label{eq:omegaW}
\omega_W = i_W^*d\Theta.
\ee
\begin{defn}[Canonical Symplectization]\label{defn:canonical-symplectization}
Let $(Q,\xi)$ be a co-oriented contact manifold. We call $(W,\omega_W)$ the
\emph{canonical symplectization} of the contact manifold $(Q,\xi)$.
\end{defn}
One important point of this canonical symplectization is the fact that it depends only on the
contact structure $\xi$ and the orientation of $Q$,
but does not depend on the choice of contact form $\lambda$. The symplectic form $\omega_W$
provides a natural subspace
$$
\widetilde \xi: = \left\{\eta \in T_\alpha W \, \Big\vert\,  \omega_W\left(\eta, \frac{\del}{\del r}\right) = 0 \right\}.
$$
Note that the projection $d_\alpha\pi_w$ restricts to an isomorphism $\widetilde \xi \to \xi$.

As explained in \cite{oh:contacton}, a choice of contact form $\lambda$ defines the map
$$
\widehat \psi: Q \times \R \to W; \quad \widehat \psi(x,s) = \psi(s) \, \lambda(x) \in T^*_xQ
$$
associated to each monotonically increasing function $\psi$ such that
\be\label{eq:varphi}
\psi(s) =
\begin{cases} 1 \quad & \mbox{for $s \geq R_1$} \\
\frac{1}{2} \quad & \mbox{for $s \leq R_0$}
\end{cases}
\ee
for any pair $R_0 < R_1$ of real numbers. We measure the symplectic area
of the composition $\widehat\psi \circ w: \dot \Sigma \to W$
for all possible variations of such $\psi$.  Hofer's original definition of
this type of energy then can be expressed as the integral
\bea
E_\CC(u)& : = & \sup_{\psi} \int_{\dot \Sigma} (\widehat\psi \circ u)^*\omega_W \nonumber\\
& = & \sup_{\psi} \int_{\dot \Sigma} d (\psi(s)\, \pi^*\lambda) \label{eq:hofer's-energy}\\
& = & \sup_{\psi}\left(\int_{\dot \Sigma} \psi(a) dw^*\lambda + \psi'(a)\, da \wedge w^*\lambda\right).
\eea
Following \cite{behwz}, we split this energy into two parts,
one purely depending on $w$
$$
E^\pi(w)= \int_{\dot \Sigma} dw^*\lambda
$$
and the other
$$
E^\lambda(u) =  \sup_{\psi} \int_{\dot \Sigma} \psi'(a)\, da \wedge w^*\lambda
= \sup_{\psi} \int_{\dot \Sigma} d(u^*\psi) \wedge u^*(\pi_Q^*\lambda).
$$
\begin{rem} The upshot of the above discussion is that Hofer's energy is
closely tied to the symplectic area of the composition $\widetilde \psi \circ u$
measured with respect to the canonical symplectic form on $W \subset T^*Q$.
Since the choice of the positive function satisfying \eqref{eq:varphi} is not
unique, we take the supremum over all such function $\psi$ to extract $\psi$-independent
quantity, which is precisely the way how Hofer's energy is defined.
\end{rem}

\section{$\lcs$ instantons and their energy}
\label{sec:lcs-instanton}

\subsection{$\lcs$ instanton equation}

In this section, we give a useful representation of the equation $\delbar_J u = 0$
for any $\lambda$-admissible almost complex structure $J$ in terms of the
contact instanton studied in \cite{oh-wang:CR-map1, oh-wang:CR-map2}.
We denote by
$$
\delbar^\pi w: = \frac{d^\pi w + J d^\pi w \circ j}{2}
$$
(similarly for $\del^\pi w$).

\begin{prop}\label{prop:lcs-instanton} Let $J$ be $\lambda$-admissible almost complex structure. Then
a map $u$ is $J$-holomorphic if and only if $(w,f)$ satisfies
\be\label{eq:lcs-instanton}
\delbar^\pi w = 0, \quad w^*\lambda \circ j = f^*d\theta.
\ee
\end{prop}
\begin{proof} By definition, we have
$\delbar_J u= \frac{du + J du j}{2}$. Writing
\be\label{eq:du}
du = \pi_{\widetilde \xi}\circ du \oplus \pi_\CV\circ du
\ee
and using the $J$ invariance of $\widetilde \xi$ and $\CV$, we compute
\beastar
J du & = & J (\Pi_{\widetilde \xi}\circ du + \Pi_\CV\circ du) \\
& = & J (\Pi_{\widetilde \xi}\circ du) + J(\Pi_\CV\circ du).
\eeastar
For the first summand, we can rewrite it as
\beastar
J (\Pi_{\widetilde \xi}\circ du) = J \left( (\pi_\xi \circ dw) \oplus 0 R_\lambda)\oplus 0 \frac{\del}{\del \theta}\right)
= (J d^\pi w) \oplus 0_\CV.
\eeastar
For the second summand, we have
$$
J(\pi_\CV \circ du) = J (w^*\lambda R_\lambda\oplus df).
$$
Therefore we have
\be\label{eq:Jdu}
J du = J d^\pi w \oplus J (w^*\lambda R_\lambda\oplus df).
\ee
By adding up \eqref{eq:du}, \eqref{eq:Jdu}, we obtain
$$
\delbar_J u = \delbar^\pi w \oplus
\left(\frac{(w^*\lambda R_\lambda\oplus df) + J (w^*\lambda R_\lambda\oplus df) \circ j}{2} \right).
$$
For the second summand, we derive that
$$
\frac{(w^*\lambda R_\lambda \oplus df) + J (w^*\lambda R_\lambda\oplus df) \circ j}{2} = 0
$$
is equivalent to $w^*\lambda \circ j = f^*d\theta$ by noting $J\frac{\del}{\del \theta} = R_\lambda$
and evaluating the equation
against the coordinate basis $\{\frac{\del}{\del s},\frac{\del}{\del t} \}$ of
any complex coordinate $z = s + i t$ on $\dot \Sigma$.

\end{proof}

\begin{rem}\label{rem:comparison}
Here we would like to compare the system \eqref{eq:lcs-instanton} with the contact instanton equation
\eqref{eq:contacton}
studied in \cite{oh-wang:CR-map1,oh-wang:CR-map2}. First of all in both systems, the first
equation depends only on the contact component $w: \dot \Sigma \to Q$ which satisfies
the contact Cauchy-Riemann map equation $\delbar^\pi w = 0$. In \cite{oh-wang:CR-map1,oh-wang:CR-map2},
all the local a priori elliptic estimates have been established. Furthermore asymptotic convergence is
also established for such map $w$
under the derivative bound $\|dw\|_{C^0} < C$ except the case where  the asymptotic charge
$$
Q = \int  w^*\lambda\circ j|_{t = t_0}
$$
does not vanish but its asymptotic period
$$
T = \lim_{t_0 \to \infty} \int  w^*\lambda|_{t = t_0}
$$
vanishes. Unlike that case where $w^*\lambda\circ j$ is exact as in the symplectization of $(Q,\lambda)$,
the second equation $d(w^*\lambda\circ j)=0$ does not enable us to establish the asymptotic convergence to
a closed Reeb orbit for such a case under the derivative bound $\|dw\|_{C^0} < C$.

In the current situation, $w$ is a contact instanton such that the one-form $w^*\lambda\circ j$
is not only closed but it is determined by  the $f$-component. In particular it always has
\emph{integral} asymptotic charge which is nothing but the degree of the map
$$
t \mapsto f(t, \cdot); S^1 \to S^1
$$
in the  coordinate $(\tau,t)$ in the strip-like region near the given puncture of  $\dot \Sigma$.
\end{rem}

\subsection{Hofer-type energy for $\lcs$ instantons}
\label{subsec:energy}

In the present section, we will derive some basic properties of \eqref{eq:lcs-instanton}.
A crucial point for the analysis of $\lcs$ instanton equation is to identify the
relevant geometric energy that enables us to establish the results of suitable
compactness and asymptotic convergence for a finite energy $\lcs$ instantons.
(See \cite[Section 6]{oh-wang:CR-map1} for delicacy of the issue
arising in the study of asymptotic convergence for the contact instanton equation.)

Fix a K\"ahler metric $h$ on $(\dot \Sigma, j)$.  The norm $|dw|$ of the map
$$
dw:(T \dot \Sigma,h) \to (TQ, g)
$$
with respect to the metric $g$ is defined by
$$
|dw|_g^2 := \sum_{i=1}^{2} {|dw(e_i)|_g}^2,
$$
where  $\{ e_1, e_2 \}$ is an orthonormal frame of $T \Sigma$
with respect to $h$.

We first introduce the $\xi$-component of the harmonic energy, which we call
the $\pi$-harmonic energy. This energy equals the contact area $\int w^*d\lambda$
`on shell' i.e., for any contact Cauchy-Riemann map, which satisfies $\delbar^\pi w = 0$

\begin{defn}\label{defn:pi-energy}
For a smooth map $\dot \Sigma \to Q$, we define the $\pi$-energy of $w$ by
\be\label{eq:Epi}
E^\pi(j,w) = \frac{1}{2} \int_{\dot \Sigma} |d^\pi w|^2.
\ee
\end{defn}

As mentioned in subsection \ref{subsec:canonical} in the context of symplectization, the $\pi$-harmonic energy
itself is not enough for the crucial bubbling-off analysis needed for the equation
\eqref{eq:lcs-instanton}. This is only because the bubbling-off analysis
requires the study of asymptotic behavior of contact instantons on the complex place $\C$.
A crucial difference between the case of contact instantons from Gromov's theory of
pseudoholomorphic curves on symplectic manifolds is that there is no removal singularity
result of the type of harmonic maps (or pseudoholomorphic maps) \emph{under the finite
$\pi$-energy condition}. Because of this, one
needs to examine the vertical part (i.e., the $R_\lambda$-component) of energy to control the asymptotic behavior of
contact instantons near the puncture. For this purpose, the Hofer-type energy
will be again crucial. In this section, we introduce this energy
in the general context of $\mathfrak{lcs}$-fication of $(Q,\lambda)$.

Now the definition of the vertical part of energy, which we call the $\lambda$-energy,
is in order. We need some preparation before giving the definition.

Let $u = (w,f):\dot \Sigma \to Q \times S^1$ be an $\lcs$ instanton.
Denote by $\overline \Sigma$ the real blow-up of the punctured Riemann
surface $\dot \Sigma$.

\begin{prop}\label{prop:S1class} Let $u: \dot \Sigma \to Q \times S^1$ be an $\lcs$ instanton
satisfying
\be\label{eq:pienergy-C1bound}
E^\pi(u) < \infty, \quad \|du\|_{C^0} < \infty.
\ee
Then the differential from $f^*d\theta$ extends smoothly to $\del \overline \Sigma$.
\end{prop}

\begin{proof} We note that the given hypothesis for $u = (w,f)$ makes
$w$ a contact instanton satisfying Hypothesis \ref{hypo:basic} and so we apply
Theorem \ref{thm:subsequence} and Proposition \ref{cor:tangent-convergence} to $w$.
\end{proof}

Recalling the isomorphism
$$
[\dot \Sigma,S^1] \cong H^1(\dot \Sigma,\Z)
$$
we may also regard the cohomology class $[u]_{S^1}$ as an element in $[\dot \Sigma,S^1]$.
This enables us to define an element in the set of homotopy classes $[\dot \Sigma,S^1]$, which
we also denote by $\eta = \eta_u$. In fact, the isomorphism
$[\dot \Sigma, S^1] \cong H^1(\dot \Sigma;\Z)$
is directly induced by the period map
$$
[f] \mapsto [f^*d\theta].
$$
Motivated by this theorem, we define the set of \emph{period classes} $\eta \in H^1(\dot \Sigma,\Z)$.

\begin{defn}[Period map and the charge class] Let $f: \dot \Sigma \to Q \times S^1$ be a smooth map.
\begin{enumerate}
\item We call the map
$$
C^\infty(\dot \Sigma,Q) \to H^1(\dot \Sigma,\Z); \quad f \mapsto [f^*d\theta]
$$
the \emph{period map} and call the cohomology class $[f^*d\theta]$
the \emph{charge class} of the map $f$.
\item For an $\text{\rm lcs}$ instanton $u = (w,f): \dot \Sigma \to Q \times S^1$, we call the cohomology class
$$
\left[f^*d\theta\right] \in H^1(\dot \Sigma,\Z)
$$
the \emph{charge class} of $u$ and write
$$
[u]_{S^1}: = [f^*d\theta].
$$
\end{enumerate}
\end{defn}

Now we consider the maps $u = (w,f)$ with a fixed charge class $\eta = \eta_u$.

\begin{prop}\label{prop:harmonic-form}
Let $h$ be any K\"ahler metric on $(\dot \Sigma,j)$ such that $h = d\tau^2 + dt^2$
with respect to the strip-like coordinates $(\tau,t)$ near each puncture of $\dot \Sigma$.
Suppose that $u = (w,f)$ satisfies \eqref{eq:pienergy-C1bound}.
Then there exists a harmonic one form $\beta_\eta$ such that
\be\label{eq:dtildef}
f^*d\theta = \beta_\eta + d \widetilde f
\ee
on $\dot \Sigma$ for a function $\widetilde f: \dot \Sigma \to \R$ that continuously extends to a function
on the real-blow up $\overline{\Sigma}$.
\end{prop}
\begin{proof}
Let $\eta \in H^1(\dot \Sigma, \Z)$ be the class represented by $f^*d\theta$. We first recall
that the one-form $f^*d\theta = w^*\lambda \circ j$ converges to $- Q\, dt + T\, d\tau$ exponentially fast
 as $\tau \to \infty$ in the given strip-like coordinate $(\tau,t)$. Therefore we can write
$$
f^*d\theta = \beta_0 + dg
$$
for some function $g$ and a closed one-form $\beta_0$ in class $\eta \in H^1(\dot \Sigma,\Z)$ which satisfy
the following:
\begin{itemize}
\item $dg \to 0$ exponentially fast as $\tau \to \infty$,
\item $\beta_0 \equiv  - Q\, dt + T\, d\tau$ for $\tau \geq R_0$ for sufficiently large $R_0> 0$.
\end{itemize}
We note that $\beta_0$ is a harmonic form on $[R_0, \infty) \times [0,1]$.
We apply the above discussion to each puncture of $\dot \Sigma$ and consider the compact domain
$$
\Sigma_{\vec R_0}: = \dot \Sigma \setminus \left( \sqcup_{i=1}^k  [R_{0,i}+ \delta, \infty) \times [0,1]\right)
$$
for some small $\delta > 0$.
We write $(\beta_{0,i}, g_i)$ for the above pair $(\beta_0,g)$ associated to the strip
$$
Z_i = [R_{0,i}, \infty)\times [0,1].
$$
We have $h = d\tau^2 + dt^2$ on each $[R_{0,i}, R_{0,i} + \delta] \times [0,1]$
for sufficiently large $R_{0,i}$'s.

Then we take the double of $(\Sigma_{\vec R_0},h)$ given by
$$
(\Sigma_{\vec R_0}^{db},h^{db}) = (\Sigma_{\vec R_0},h) \# (\Sigma_{\vec R_0},h)^{op}
$$
where $\Sigma_{\vec R_0}^{op}$ is $\Sigma_{\vec R_0}$ with opposite orientation.
We also have the canonical double of the form $\beta_0$ on $\Sigma_{\vec R_0}^{db}$ which we denote by
$\beta_0^{db}$.

This one-form $\beta_0^{db}$ is closed on $\Sigma^{db}$ and so naturally defines a (de Rham) cohomology class
in $H^1(\Sigma_{\vec R_0}^{db};\Z)$.
Let $\beta$ be the unique harmonic one-form representing the class $\eta^{db}$ on the double $\Sigma_{\vec R_0}^{db}$.
Then $\beta$ is invariant
under the natural involution
$\iota: \Sigma_{\vec R_0}^{db} \to \Sigma_{\vec R_0}^{db}$, i.e., $\iota^*\beta = \beta$
since $\iota$ is an isometry.

On the annulus $[-\delta, \delta]_i \times S^1  \subset [-\delta, \delta] \times \del \Sigma_{\vec R_0}$, we write
$$
\beta = c(\tau,t)\, d\tau + d(\tau,t)\, dt
$$
and  consider the complex-valued function
$$
\kappa(\tau + it) = (c(\tau + i t) + T) + \sqrt{-1} (d(\tau + it) - Q)
$$
as a $t$-periodic complex-valued function on $\C$.
Then since $\beta$ is a harmonic one-form, the function $\kappa$ is a holomorphic function which also satisfies
\be\label{eq:Im}
\int_{\{\tau\} \times S^1} \Im \kappa (\tau,t)\, dt = \int_{\{\tau\} \times S^1} \beta  - Q \equiv 0
\ee
for all $\tau$. Furthermore, the identity $\iota^*\beta = \beta$ implies
\be\label{eq:Re}
\kappa(-\tau + it) = - (c(\tau + i t) + T) + \sqrt{-1} (d(\tau + it)- Q).
\ee
In particular, substituting $\tau = 0$, we get $c(i t) +T = - (c(i t) + T)$ which implies
both sides vanish for all $t \in S^1$.
Combining \eqref{eq:Im} and \eqref{eq:Re}, we have proved $\kappa(\tau + it) \equiv 0$ on each $[-\delta, \delta]_i \times S^1$
for $i = 1, \cdots, k$. But this is equivalent to the equality
$$
\beta(\tau,t) = T\, d\tau - Q\, dt
$$
on the strips $[-\delta, \delta]_i \times S^1$.

Now the restriction of $\beta$ to $\Sigma_{\vec R_0}$ automatically extends to
$\dot \Sigma$ by setting it to be $\beta = T d\tau - Q\, dt$ on $[R_{0,i} + \delta, \infty) \times [0,1]$.
We denote the resulting harmonic form still by $\beta$ on $\dot \Sigma$. Then
it satisfies $[\beta] = [f^*d\theta]$ and
$$
\beta - \beta_0 = 0
$$
on $[R_{0,i} + \delta, \infty) \times [0,1]$ and hence
$$
\beta = f^*d\theta + d\widetilde f
$$
such that $d(\widetilde f - g) \to 0$ exponentially fast. This finishes the proof.
\end{proof}

In particular on the given cylindrical neighborhood $D_\delta(p) \setminus \{p\}$,
$d \widetilde f \to 0$ as $\tau \to \infty$ by the asymptotic convergence theorem
Theorem \ref{thm:subsequence}.
We remark that when $w$ is given, the function $\widetilde f$ on $\dot \Sigma$ is uniquely determined
by the equation \eqref{eq:dtildef} modulo
the shift by a constant.

By construction, \eqref{eq:lcs-instanton} for $u = (w,f)$ in class
$
[u]_{S^1}= \eta  \text{ in }\, H^1(\dot \Sigma)
$
is equivalent to the following
\be\label{eq:abbas}
\begin{cases} \delbar^\pi w = 0,\\
w^*\lambda\circ j - d \widetilde f = \beta\\
\Delta \beta = 0
\end{cases}
\ee
for the triple $(w,\widetilde f, \beta)$ which determines (modulo addition by $\theta_0$) the component
$f: \dot \Sigma \to S^1$ by the equation
$$
\beta + d \widetilde f= f^*d\theta, \quad [\beta] = \eta \text{ in }\, H^1(\dot \Sigma).
$$
\begin{rem}
 When $g(\dot \Sigma) = 0$, the harmonic one-form $\beta$ is uniquely determined by
its asymptotic charges at the punctures. In other words, the following asymptotic boundary
value problem has the unique solution
$$
\begin{cases} \Delta \beta = 0, \\
\lim_{\tau \to \pm \infty} \epsilon_j^*\beta = \pm(-Q_j)\,dt
\end{cases}
$$
on $\dot \Sigma$ where $\epsilon_j: [0,\infty) \times S^1 \to \dot \Sigma$
(or $\epsilon_j: (-\infty,0] \times S^1 \to \dot \Sigma$ is the
strip-like coordinates at the puncture $r_j$. When $g(\dot \Sigma) > 0$, then the solution space
for this asymptotic boundary value problem has dimension $g(\dot \Sigma)$.
\end{rem}

We are now ready to give the definition of $\lambda$-energy. Denoting
$\varphi = \psi'$ for the function $\psi$ given in section \ref{subsec:canonical}, we introduce the following class of
test functions
\begin{defn}\label{defn:CC} We define
\be
\CC =\{\varphi: \R \to \R_{\geq 0} \mid \supp \varphi \, \text{is compact}, \, \int_\R \varphi = 1\}.
\ee
\end{defn}
Later for the purpose of compactification of relevant moduli spaces
$\CM_{k,\ell}(\dot \Sigma, Q \times S^1;J;(\vec \gamma^+,\vec \gamma^-))$,
we need to establish a uniform upper bound for the energy of
$$
u \in  \CM_{k,\ell}(\dot \Sigma, Q \times S^1;J;(\vec \gamma^+,\vec \gamma^-)),
$$
which is an
important first step towards compactification of
$\CM_{k,\ell}(\dot \Sigma, Q \times S^1;J;(\vec \gamma^+,\vec \gamma^-))$.

\begin{defn}[$E_{\CC,\eta}$-energy] Let $\eta \in H^1(\dot \Sigma, \Z)$ be given.
Let $w$ satisfy $w^*\lambda \circ j = f^*d\theta$ with $[u]_{S^1} = \eta$. Then we define
$$
E_{\CC,\widetilde f}(j,u) =  \sup_{\varphi \in \CC} \int_{\dot \Sigma} df \circ j \wedge d(\psi(\widetilde f))
= \sup_{\varphi \in \CC} \int_{\dot \Sigma} d(\psi(\widetilde f)) \wedge w^*\lambda.
$$
\end{defn}
 We note that
$$
d(\psi(\widetilde f)) \wedge w^*\lambda = \psi'(\widetilde f) d\widetilde f \wedge w^*\lambda
= \varphi(\widetilde f) w^*\lambda \circ j \wedge w^*\lambda \geq 0
$$
and hence we can rewrite $E_{\CC,\eta}(j,w)$ into
$$
E_{\CC,\widetilde f}(j,u) = \sup_{\varphi \in \CC} \int_{\dot \Sigma} \varphi(\widetilde f) d\widetilde f \wedge w^*\lambda.
$$
\begin{prop}\label{prop:a-independent} Let $u = (w,f)$ satisfy
$w^*\lambda \circ j = f^*d\theta$. If $g$ is another function satisfying $f^*d\theta = g^*d\theta$,
then we have $E_{\CC,\widetilde f}(w) = E_{\CC,\widetilde g}(w)$ for any lifting $(\widetilde f,\widetilde g)$ whenever
$d\widetilde f = w^*\lambda\circ j = d\widetilde g$
on $\dot \Sigma$.
\end{prop}

\begin{proof} We first note that $g(z) = f(z) + c$ for some constant $c$ on each connected component of
$\dot \Sigma$.
Certainly $d\widetilde f$ or $w^*\lambda$ are independent of the addition of the constant $c$.
On the other hand, we have
$$
\varphi(\widetilde g) = \varphi(\widetilde f + c)
$$
and the function $a \mapsto \varphi(a + c)$ still lie in $\CC$. Therefore after taking
the supremum over $\CC$, we have derived
$$
E_{\CC,\widetilde f}(j,u) = E_{\CC,\widetilde g}(j,u).
$$
This finishes the proof.
\end{proof}

This proposition enables us to introduce the following

\begin{defn}[Vertical energy]\label{defn:vertical-energy} Let $\eta \in H^1(\dot \Sigma,\Z)$ be given.
We denote the common value of $E_{\CC,\eta}(j,u)$ by $E^\perp_\eta(j,u)$,
and call it the \emph{vertical energy in charge class $\eta$}.
\end{defn}

The following then would be the definition of the total energy.

\begin{defn}[Total energy] Let $w:\dot \Sigma \to Q \times S^1$ be any smooth map
in class $[u]_{S^1} = \eta$.
We define the total energy of $u$ by
\be\label{eq:total-energy}
E(j,u) = E^\pi(j,u) + E^\perp_\eta(j,u)
\ee
\end{defn}

We now provide examples of finite energy $\lcs$ instanton \emph{with nonzero charge}.
\begin{exm} Consider the contact manifold $Q = S^3$ with the standard
contact structure (i.e. the $CR$-structure of $S^3 \subset \C^2$)
and take the standard contact form $\lambda$ on $S^3$
$$
\lambda = \frac12 \sum_{i=1}^2 (x_i dy_i - y_i dx_i).
$$
Then the Reeb foliation is the one given by Hopf circles. Let $C$ be any
one of them. Now consider the $\mathfrak{lcs}$-fication $M=S^1 \times S^3$ and the
Reeb torus $T: = S^1 \times C \subset M$. We fix a parameterization of $C$ by
$\gamma: [0,1]/\sim \to C$ and that of $T$ by
$$
[0,1]^2 \to T; \quad (s,t) \mapsto (s, \gamma(t)).
$$
Let $0 \neq Q  \in \Z$ be any given integer and consider
the map
$
u: \R \times S^1 \to M
$
defined by
$$
u(\tau,t) = (t, \gamma(Q\tau + t)).
$$
It is easy to check that $u$
is an $\lcs$ instanton with its charge
$$
Q \in H^1(\R \times S^1) \cong H^1(S^1) \cong \Z.
$$
The total energy of this $\lcs$ instanton is zero.
\end{exm}

\section{$\lcs$ instantons on the plane}
\label{sec:onC}

As in Hofer's bubbling-off analysis in pseudo-holomorphic curves on symplectization \cite{hofer}, it
turns out that study of contact instantons on the plane plays a crucial role in the bubbling-off
analysis of contact instantons too.

For this purpose, we start with a proposition which is an analog to
Theorem 31 \cite{hofer}. Our proof is a slight modification and some simplification
of Hofer's proof
of Theorem 31 \cite{hofer} in our generalized context.

\begin{prop}\label{prop:C^1}
Let $u: \C \to Q \times S^1$ be an $\lcs$ instanton. Regard $\infty$ as
a puncture of $\C  = \C P^1\setminus \{\infty\}$. Suppose $|dw|_{C^0} < \infty$ and
\be\label{eq:asymp-densitybound}
E^\pi(u)=0, \quad  E^\perp(u) < \infty.
\ee
Then $u$ is a constant map.
\end{prop}
\begin{proof} Let $\dot \Sigma \cong \C P^1 \setminus \{\infty\} \cong \C$. Then
$H^1(\dot \Sigma,\Z) = 0$ and hence the asymptotic charge $Q = 0$ in the cylindrical
coordinate $(\tau,t)$ near $\infty$.

Therefore the one-form $f^*d\theta$ on $\C$ is exact so that
$$
w^*\lambda \circ j = f^*d\theta = d\widetilde f.
$$
We have only to look at the $w$ component since constancy of $f$ is determined by the equation
$w^*\lambda \circ j= f^*d\theta$.

From the equality $\frac12 |d^\pi w|^2\, dA = d(w^*\lambda)$ and the hypothesis $E^\pi(w) = 0$,
we infer $|d^\pi w|^2 = 0 = d(w^*\lambda)$ in addition to $d(w^*\lambda \circ j) = 0$.
Therefore we derive that $d^\pi w = 0$. This implies
$$
dw = w^*\lambda\otimes R_\lambda(w)
$$
with $w^*\lambda$ a bounded harmonic one-form. The boundedness of $w^*\lambda$ follows from the hypothesis
$|dw|_{C^0} < \infty$. Since $\C$ is connected, the image of $w$ must be contained
in a single leaf of the Reeb foliation. We parameterize the leaf by $\gamma: \R \to Q$,
$\gamma = \gamma(t)$.

Then there is a smooth function $b = b(z)$ such that
$$
w(z) = \gamma(b(z)).
$$
Since $w^*\lambda$ is exact on $\C$, $w^*\lambda = db$ for some function $b$. Since we also have $d(w^*\lambda\circ j) = 0$,
$$
d(db \circ j) = 0
$$
i.e., $b: \C \to \R$ is a harmonic function and hence $b$ is the imaginary part of a holomorphic
function $f$, i.e., $f(z) = a(z) + ib(z)$. Since $b$ has bounded derivative, the gradient of $f$
is also bounded on $\C$. Therefore $f(z) = \alpha z + \beta$ for some constants $\alpha, \, \beta \in \C$.

Once this is achieved, the rest of the argument is exactly the same as Hofer's proof of Lemma 28 \cite{hofer}
via the usage of the $\lambda$-energy bound $E^\perp(w) < \infty$ and so omitted.
\end{proof}

We recall the following useful lemma from \cite{hofer-viterbo} whose proof we refer thereto.
\begin{lem}\label{lem:Hofer-lemma} Let $(X,d)$ be a complete metric space, $f: X \to \R$ be a
nonnegative continuous function, $x \in X$ and $\delta > 0$. Then there exists $y \in X$ and
a positive number $\epsilon \leq \delta $ such that
$$
d(x,y) < 2 \delta, \, \max_{B_y(\epsilon)} f \leq 2 f(y), \, \epsilon\, f(y) \geq \delta f(x).
$$
\end{lem}

Using the above proposition, we prove the following fundamental result.

\begin{thm}\label{thm:C1bound} Let $u: \C \to Q \times S^1$ be an $\lcs$ instanton.
Suppose
\be\label{eq:Epi-bound}
E(u) = E^\pi(u) + E^\perp(u) < \infty.
\ee
Then $|du|_{C^0} < \infty$.
\end{thm}
\begin{proof} Again it is enough to establish $|dw|_{C^0} < \infty$.
Suppose to the contrary that $|dw|_{C^0} = \infty$ and let $z_\alpha$ be a blowing-up
sequence. We denote $R_\alpha = |dw(z_\alpha)| \to \infty$. Then by applying Lemma \ref{lem:Hofer-lemma},
we can choose another such sequence $z_\alpha'$ and $\epsilon_\alpha \to 0$ such that
\be\label{eq:blowingup-sequence}
|dw(z_\alpha')| \to \infty, \quad \max_{z \in D_{\epsilon_\alpha}(z_\alpha')}|dw(z)| \leq 2 R_\alpha,
\quad \epsilon_\alpha R_\alpha \to 0.
\ee
We consider the re-scaling maps $\widetilde w_\alpha: D^2_{\epsilon_\alpha R_\alpha}(0) \to Q$
defined by
$$
w_\alpha(z) = w \left(z_\alpha' + \frac{z}{R_\alpha}\right).
$$
Then we have
$$
|d w_\alpha|_{C^0; \epsilon_\alpha R_\alpha} \leq 2, \quad |d w_\alpha(0)|=1.
$$
Applying Ascoli-Arzela theorem, there exists a continuous map $w_\infty: \C \to Q$ such that
$ w_\alpha \to w_\infty$ uniformly on compact subsets. Then by the a priori $W^{k,2}$-estimates,
Theorem \ref{thm:Wk2}, it follows that the convergence is in the $C^\infty$-sense on every compact subset,
and $w_\infty$ is smooth. Furthermore
$w_\infty$ satisfies $\delbar^\pi w_\infty = 0 = d(w_\infty^*\lambda \circ j) = 0$,
$$
E^\pi(w_\infty), \, E^\lambda(w_\infty) \leq E(w) < \infty
$$
and
$$
|d w_\infty|_{C^0; \C} \leq 2, \quad |dw_\infty(0)|=1.
$$

On the other hand, by the finite $\pi$-energy hypothesis and density identity $\frac12 |d^\pi w|^2 \, dA = d(w^*\lambda)$,
we derive
\beastar
0 & = & \lim_{\alpha \to \infty} \int_{D_{\epsilon_\alpha}(z_\alpha')} d(w^*\lambda) =
\lim_{\alpha \to \infty} \int_{D_{\epsilon_\alpha R_\alpha}(z_\alpha')} d( w_\alpha^*\lambda)\\
&= & \lim_{\alpha \to \infty} \int_{D_{\epsilon_\alpha R_\alpha}(z_\alpha')}\frac12 |d^\pi \widetilde w_\alpha|^2
= \int_\C \frac12 |d^\pi w_\infty|^2.
\eeastar
 Therefore we derive
$$
E^\pi(w_\infty) = 0.
$$
Then Proposition \ref{prop:C^1} implies $w_\infty$ is a constant map which contradicts
$|dw_\infty(0)| = 1$. This finishes the proof.
\end{proof}

An immediate corollary of this theorem and Proposition \ref{prop:C^1} is the following

\begin{cor}\label{cor:pi-positive} For any non-constant $\lcs$ instanton $u: \C \to Q \times S^1$
with energy bound $E(u) < \infty$, we obtain
$$
E^\pi(w) = \int z^*\lambda > 0
$$
for $z = \lim_{R \to \infty} w(R e^{2\pi it})$. In particular $E^\pi(w) \geq T_\lambda > 0$.
\end{cor}

Combining Theorem \ref{thm:subsequence}, Theorem \ref{thm:C1bound} and  Proposition \ref{prop:S1class},
we immediately derive

\begin{cor} Let $u = (w,f)$ be a non-constant $\lcs$ instanton on $\C$ with
\be\label{eq:C1-densitybound}
E(u) < \infty.
\ee
Then there exists a sequence $R_j \to \infty$ and a Reeb orbit $\gamma$ such that
$z_{R_j} \to \gamma(T(\cdot))$ with $T \neq 0$ and
$$
T = E^\pi(w), \quad Q = \int_z w^*\lambda \circ j = 0.
$$
\end{cor}
\begin{proof} Since $[\overline{\C},S^1] = 0$, it follows $Q = 0$.

If $T = 0$, the above theorem shows that there exists a sequence
$\tau_i \to \infty$ such that $w(\tau_i,\cdot)$ converges to a constant in the $C^\infty$
topology and so
$$
\int_{\{\tau = \tau_i\}} w^*\lambda \to 0
$$
as $i \to \infty$. By Stokes' formula, we derive
$$
\int_{D_{e^{\tau_i}}(0)} w^*d\lambda = \int_{\tau = \tau_i} w^*\lambda \to 0.
$$
On the other hand, we have
$$
E^\pi(w) = \lim_{i \to \infty} \int_{D_{e^{\tau_i}}(0)} \frac12 |d^\pi w|^2
 = \lim_{i \to \infty} \int_{D_{e^{\tau_i}}} w^*d\lambda = 0.
$$
This contradicts Corollary \ref{cor:pi-positive}, which finishes the proof.
\end{proof}

The following is the analog to Proposition 30 \cite{hofer}.

\begin{cor}\label{cor:C^1oncylinder} Let $u$ be an $\lcs$ instanton on
$\R \times S^1$ in class $\eta = [u]_{S^1}$ with $E(u) < \infty$.
Then $\|du\|_{C^0} < \infty$.
\end{cor}
\begin{proof} By definition, we have a harmonic one-form $\beta_\eta$ and
a function $\widetilde f: \R \times S^1 \to \R$ satisfying
$$
f^*d\theta = \beta_\eta + d\widetilde f.
$$
We note that $H^1(\dot \Sigma,\Z) \cong \Z$ and
$\beta_\eta(\tau,t) = a\, d\tau + b\, dt$ for constants $a, \, b$ and
$d\widetilde f \to 0$ as $\tau \to \pm \infty$. In particular $|\beta_\eta(\tau,t)|_{C^0} = \sqrt{a^2 + b^2}$.
In particular, the pair $\widetilde u= (\widetilde f, w): \dot \Sigma \to Q \times \R$ satisfies
$$
\delbar^\pi w = 0, \quad f^*d\theta - d\widetilde f = \beta_\eta.
$$
In other words, the pair $(\widetilde f, w)$ satisfies the following perturbed
Cauchy-Riemann equation $\delbar_J \widetilde u = \beta_\eta$ for a given harmonic one-form
on $\R \times S^1$ which is independent of $\widetilde u$, an inhomogeneous
$\delbar_J$-equation (with constant one-form).

As in Hofer's proof of Proposition 30 \cite{hofer}, using the Hofer's energy for $\widetilde u$,
we can apply the same kind of bubbling-off argument as that of Theorem \ref{thm:C1bound}
and derive the same conclusion.
\end{proof}

We now derive the following from Proposition \ref{prop:S1class}

\begin{cor}\label{cor:periodoncylinder} For any $\lcs$ instanton on
$\R \times S^1$ with $E(u) < \infty$, the charge class
$[u]_{S^1}$ is well-defined.
\end{cor}

\section{Bubbling-off analysis and the period-gap theorem }
\label{sec:e-regularity}

We recall some basic definitions and results from
from \cite{oh-wang:CR-map1}, \cite{oh:contacton} on contact instantons
which are also relevant to the current study of lcs-instantons.
In \cite{oh-wang:CR-map1}, the local a priori $W^{k,2}$-regularity estimates are established with respect to
the bounds of $\|dw\|_{L^4}$ and $\|dw\|_{L^2}$. Therefore in addition to the local a priori $W^{k,2}$-regularity estimates,
one should establish another crucial ingredient, the $\epsilon$-regularity result, for the
study of moduli problem as usual in any of conformally invariant geometric non-linear PDE's.
This will in turn establish the $W^{1,p}$-bound
with $p > 2$ (say $p = 4$). (See \cite{sacks-uhlen}.)

In the current setting of $\lcs$ instanton map, it is not obvious what would be
the precise form of relevant $\epsilon$-regularity statement is.
We formulate this $\epsilon$-regularity theorem in the setting of contact instantons.

\begin{defn} Let $\lambda$ be a contact form of contact manifold $(Q,\xi)$.
Denote by $\frak R eeb(Q,\lambda)$ the set of closed Reeb orbits.
We define $\operatorname{Spec}(Q,\lambda)$ to be the set
$$
\operatorname{Spec}(Q,\lambda) = \left\{\int_\gamma \lambda \mid \lambda \in \frak Reeb(Q,\lambda)\right\}
$$
and call it the \emph{action spectrum} of $(Q,\lambda)$. We denote
$$
T_\lambda: = \inf\left\{\int_\gamma \lambda \mid \lambda \in \frak Reeb(Q,\lambda)\right\}.
$$
\end{defn}

We set $T_\lambda = \infty$ if there is no closed Reeb orbit.
The following is a standard lemma in contact geometry

\begin{lem} Let $(Q,\xi)$ be a closed contact manifold. Then
$\operatorname{Spec}(Q,\lambda)$ is either empty or a countable nowhere dense subset of $\R_+$
and $T_\lambda > 0$. Moreover the subset
$$
\operatorname{Spec}^{K}(Q,\lambda) = \operatorname{Spec}(Q,\lambda) \cap (0,K]
$$
is finite for each $K> 0$.
\end{lem}

The constant $T_\lambda$ will enter in a crucial way in the following
period gap theorem.  The proof of this theorem will closely follow the argument used in
\cite[Section 8.4]{oh:book} and \cite{oh:imrn} by adapting it to
the proof of the current gap theorem with the replacement of
the standard harmonic energy by the $\pi$-harmonic energy.

\begin{thm}[Theorem 7.4 \cite{oh:contacton}]\label{thm:e-regularity}
Denote by $D^2(1)$ the closed unit disc and let $u = (w,f)$ be an $\text{\rm lcs}$ instanton
defined on $D^2(1)$ so that $w:D^2(1)\to Q$ satisfies
$$
\delbar^\pi w = 0, \, w^*\lambda\circ j = f^*d\theta.
$$
Assume the vertical energy bound $E^\perp (w) < K_0$ defined in Definition \ref{defn:vertical-energy}.
Then for any given $0 <\epsilon < T_\lambda$ and $w$ satisfying
$E^\pi(w) < T_\lambda - \epsilon$, and for a smaller disc $D' \subset \overline D' \subset D$,
there exists some $K_1 = K_1(D', \epsilon,K_0) > 0$
\be\label{eq:dwC0}
\|dw\|_{C^0;D'} \leq K_1
\ee
where $K_1$ depends only on $(Q,\lambda,J)$, $\epsilon$, $D' \subset D$.
\end{thm}
\begin{proof}
Suppose to the contrary that
there exists a disc $D' \subset D$ with $\overline{D'} \subset \overset{\circ}D$ and a sequence $\{ w_\alpha \}$
such that
$$
\delbar^\pi w_\alpha = 0, \quad w_\alpha \circ j = f_\alpha^*d\theta
$$
and
\be\label{eq:2to0ptoinfty}
E^\pi_{\lambda,J;D}(w_\alpha) < T_\lambda - \epsilon,, E^\perp(w_\alpha) < K_0,
\quad \norm{dw_\alpha}{C^0,D'} \to \infty
\ee
as $\alpha \to \infty$. Let $x_\alpha \in D'$ be such that $|dw_\alpha(x_\alpha)| \to \infty$.
By choosing a subsequence, we may assume that $x_\alpha \to x_\infty\in \overline D' \subset \overset{\circ} D$.
We take a coordinate chart centered at $x_\infty$ on $D_{x_\infty}(\delta) \subset \overset{\circ} D$
and identify $D_{x_\infty}(\delta)$ with the disc $D^2(\delta) \subset \C$ and $x_\infty$ with $0 \in \C$.
This can be done by choosing $\delta > 0$ sufficiently small since we assume $\overline D' \subset \overset{\circ} D$.
Then $x_\alpha \to 0$. We choose $\delta_\alpha \to 0$ so that $\delta_\alpha |dw_\alpha(x_\alpha)| \to \infty$.

Using Lemma \ref{lem:Hofer-lemma} as before, we adjust the sequence $x_\alpha$ to $y_\alpha$ so that
$y_\alpha \to 0$ and
\be\label{eq:adjustedy}
\max_{x \in B_{y_\alpha}(\epsilon_\alpha)}|dw_\alpha| \leq 2|dw_\alpha(y_\alpha)|, \quad
\delta_\alpha |dw_\alpha(y_\alpha)| \to \infty.
\ee
We denote $R_\alpha =  |dw_\alpha(y_\alpha)|$ and consider the re-scaled map
$$
v_\alpha(z) = w_\alpha\left(y_\alpha + \frac{z}{R_\alpha}\right).
$$
Then the domain of $w_\alpha$ at least includes $z \in \C$ such that
$$
y_\alpha + \frac{z}{R_\alpha} \in D^2(\delta),
$$
i.e., those $z$'s satisfying
$$
\left|y_\alpha + \frac{z}{R_\alpha} \right| \leq \delta.
$$
In particular, if $|z| \leq R_\alpha (\delta - |y_\alpha|)$, $v_\alpha(z)$ is
defined. Since $y_\alpha \to 0$ and $\delta_\alpha \to 0$ as $\alpha \to \infty$,
$R_\alpha (\delta - |y_\alpha|)> R_\alpha \epsilon_\alpha$
eventually, $v_\alpha$ is defined on $D^2(\epsilon_\alpha R_\alpha)$ for all sufficiently
large $\alpha$'s.
Since $\delta_\alpha R_\alpha \to \infty$ by \eqref{eq:adjustedy}, for any given $R>0$,
$D^2(\delta_\alpha R_\alpha)$ of $v_\alpha (z)$ eventually contains $B_{R+1}(0)$.

Furthermore, we may assume,
$$
B_{R+1}(0) \subset \left\{ z \in \mathbb{C} \mid \eta_\alpha z + y_\alpha \in \overline{D}'\right\}
$$
Therefore, the maps
$$
v_\alpha : B_{R+1} (0) \subset \mathbb{C} \to M
$$
satisfy the following properties:
\begin{enumerate}
 \item[(i)]  $E^\pi(v_\alpha) < T_\lambda -\epsilon$, \, $\delbar^\pi v_\alpha = 0$, \,
 $E^\perp(v_\alpha) \leq K_0$,
 (from the scale invariance)
 \item[(ii)] $|dv_\alpha(0)|=1 $ by definition of $v_\alpha$ and $R_\alpha$,
 \item[(iii)]$\norm{dv_\alpha}{C^0,B_1(x)} \leq 2$ for all $x \in B_R(0) \subset D^2(\epsilon_\alpha R_\alpha)$,
 \item[(iv)] $\delbar^\pi v_\alpha =0$ and $d(v_\alpha^*\lambda \circ j) = 0$.
\end{enumerate}

For each fixed $R$, we take the limit of $v_\alpha|_{B_R}$, which we denote by $w_R$.
Applying (iii) and then the local $W^{k,2}$ estimates, Theorem \ref{thm:Wk2},  we obtain
$$
\norm{dv_\alpha}{k,2;B_{\frac9{10}}(x)} \leq C
$$
for some $C=C(R)$. By the Sobolev embedding theorem, we have a
subsequence that converges in $C^2$ in each $B_{\frac{8}{10}}(x), x \in \overline D'$. Then
we derive that the convergence is in $C^2$-topology on $B_{\frac{8}{10}}(x)$ for all $x \in \overline D'$ and
in turn on $B_R(0)$.

Therefore the limit $w_R: B_R(0) \to M$ of $v_\alpha|_{B_R(0)}$ satisfies
\begin{itemize}
\item[(1)] $E^\pi(w_R) \leq T_\lambda -\epsilon$, $\delbar^\pi w_R = 0$,
$d(w_R^*\lambda \circ j) = 0$ and $E^\perp(v_\alpha) \leq K_0$,
\item[(2)] $E^\pi(w_R) \leq \limsup_\alpha  E^\pi_{(\lambda,J;B_R(0))}(v_\alpha) \leq T_\lambda -\epsilon$,
\item[(3)]  Since $v_\alpha \to w_R$ converges in $C^2$, we have
$$
\norm{dw_R}{p,B_1(0)}^2 = \lim_{\alpha \to \infty} \norm{dv_\alpha}{p,B_1(0)}^2 \geq \frac{1}{2}.
$$
\end{itemize}
By letting $R \to \infty$ and taking a diagonal subsequence argument, we have
derived a nonconstant contact instanton map $w_\infty: \C \to Q$. Therefore by definition of $T_\lambda$,
we must have $E^\pi(w_\infty) \geq T_\lambda$.

On the other hand, the bound
$E^\pi(w_R) \leq T_\lambda -\epsilon$ for all $R$ and again by Fatou's lemma implies
$$
E^\pi(w_\infty) \leq T_\lambda -\epsilon
$$
which gives rise to a contradiction.
This finishes the proof of \eqref{eq:dwC0}.
\end{proof}

\section{Asymptotic behaviors of finite energy $\lcs$ instantons}
\label{sec:asymptotic}

In this section, we study the asymptotic behavior of $\lcs$ instantons $u=(w,f)$
on the Riemann surface $(\dot\Sigma, j)$ associated with a metric $h$ with \emph{cylindrical ends}.
To be precise, we assume there exists a compact set $K_\Sigma\subset \dot\Sigma$,
such that $\dot\Sigma-\Int(K_\Sigma)$ is a disjoint union of punctured disks
 each of which is isometric to the half cylinder $[0, \infty)\times S^1$ or $(-\infty, 0]\times S^1$, where
the choice of positive or negative cylinders depends on the choice of analytic coordinates
at the punctures.
We denote by $\{p^+_i\}_{i=1, \cdots, l^+}$ the positive punctures, and by $\{p^-_j\}_{j=1, \cdots, l^-}$ the negative punctures.
Here $l=l^++l^-$. Denote by $\phi^{\pm}_i$ such isometries from cylinders to disks.

We start with the following lemma

\begin{lem}\label{lem:|dw|<infty} Let $\eta \in H^1(\dot \Sigma, \Z)$ be given.
Suppose $[u]_{S^1} = \eta$ and $E^\pi(w) \leq E(u) = E_\eta (u) < \infty$. Then
$$
\|du\|_{C^0;\dot \Sigma} < \infty.
$$
\end{lem}
\begin{proof}
By the finiteness $E^\pi(w)\leq E(u) < \infty$, we can choose sufficiently small $\delta > 0$ such that
$$
E^\pi(w|_{\Sigma\setminus \Sigma(\delta)})< \frac{1}{2}T_\lambda, \, E^\perp(u) < \epsilon(\delta)
$$
Denote
$$
\Sigma(\delta) = \dot \Sigma \setminus \cup_{\ell =1}^k D_{r_\ell}(\delta).
$$
Then we apply the $\epsilon$-regularity theorem
Theorem \ref{thm:e-regularity}, to $w$ on $\cup_{\ell =1}^k D_{r_\ell}(\delta) = \dot \Sigma \setminus \Sigma(\delta)$
to derive
$$
\|du\|_{C^0;\cup_{\ell =1}^k D_{r_\ell}} < \infty.
$$
Obviously $\|du\|_{C^0;\Sigma(\delta)} < \infty$ and hence the proof.
\end{proof}

Because the behavior of the component $f$ largely determined by the component $w$,
we will focus on the study of the latter's asymptotic behavior.
We now state our assumptions for the study of the behavior of punctures.

\begin{defn}Let $\dot\Sigma$ be a punctured Riemann surface with punctures
$\{p^+_i\}_{i=1, \cdots, l^+}\cup \{p^-_j\}_{j=1, \cdots, l^-}$ equipped
with a metric $h$ with \emph{cylindrical ends} outside a compact subset $K_\Sigma$.
Let
$w: \dot \Sigma \to M$ be any smooth map.
\end{defn}

Throughout this section, we work locally near one puncture, i.e., on
$D^\delta(p) \setminus \{p\}$. By taking the associated conformal coordinates $\phi^+ = (\tau,t)
:D^\delta(p) \setminus \{p\} \to [0, \infty)\times S^1 \to $ such that $h = d\tau^2 + dt^2$,
we need only look at a map $w$ defined on the half cylinder $[0, \infty)\times S^1\to M$
without loss of generality.

The following is proved in \cite{oh:contacton}.

\begin{lem}[Lemma 8.2 \cite{oh:contacton}]\label{lem:massless}
Let $\dot \Sigma$ be any punctured Riemann surface.
Suppose $w: \dot \Sigma \to Q$ is a massless contact instanton
on $\dot \Sigma$. Then $w^*\lambda $ is a harmonic 1-form and
the image of $w$ lies in a single leaf of the Reeb foliation.
\end{lem}

The following result is the lcs-counterpart of \cite[Proposition 8.3]{oh:contacton}
whose proof can be repeated verbatim and so is omitted.

\begin{prop}[Compare with Proposition 8.3 \cite{oh:contacton}]\label{prop:pole-structure}
Let $u = (w,f)$ be an $\lcs$ instanton on $\dot \Sigma$ with punctures $p \in \{p_1, \cdots, p_k\}$.
Consider the complex-valued one-form on $\dot \Sigma$ defined by
\be\label{eq:chi}
\chi: = f^*d\theta + \sqrt{-1} w^*\lambda.
\ee
Let $p \in \{p_1, \cdots, p_k\}$ and let $z$ be an analytic coordinate at $p$. Suppose
$$
E(u) < \infty.
$$
Then for any given sequence $\delta_j \to 0$ there exists a subsequence, still denoted by $\delta_j$, and a conformal diffeomorphism
$\varphi_j: [-\frac{1}{\delta_j}, \infty) \times S^1 \to D_{\delta_j}(p)\setminus \{p\}$ such that
the one form $\varphi_j^*\chi$ converges to a bounded holomorphic one-form
$\chi_\infty$ on $(-\infty, \infty) \times S^1$.
\end{prop}

We would like to emphasize that at the moment, the limiting holomorphic one-form $\chi_\infty$
may depend on the choice of subsequence.

Finally  we would like to further analyze the asymptotic behavior of the instanton $u = (w,f)$.
We will show that if both $T$ and $Q$ vanish,
$w_\infty$ is a constant map and so the puncture is removable.
For this purpose, we consider the complex valued one-form
\be\label{eq:specialchi}
\chi =  f^*d\theta + \sqrt{-1} w^*\lambda = w^*\lambda \circ j + \sqrt{-1} w^*\lambda
\ee
where the latter equality follows from $w^*\lambda \circ j = f^*d\theta$ as a
part of $\lcs$ instanton equation. We note that for $\lcs$ instanton $u = (w,f)$ $\chi$
is expressed purely in terms of $w$.

\begin{defn}\label{defn:charge} Let $(\Sigma,j)$ be a closed
Riemann surface with a finite number of marked points $\{p_1, \cdots, p_k\}$. Denote by $\dot \Sigma$ the
associated punctured Riemann surface with cylindrical metric near the punctures, and let $\overline \Sigma$
be the real blow-up of $\Sigma$ along the punctures.
Let $w$ be a contact instanton map. Let $p \in \{p_1, \cdots, p_k\}$. We call the integrals
\bea
Q(p) &: =& \int_{\del_{\infty;p}\Sigma} w^*\lambda \circ j\\
T(p) &: =& \int_{\del_{\infty;p}\Sigma} w^*\lambda
\eea
the \emph{contact instanton charge} and \emph{contact instanton action} at $p$ respectively.
Here $\del_{\infty;p}\Sigma$ is the
boundary component corresponding to $p$ of the real blow-up $\overline \Sigma$ of $\dot \Sigma$.
Then we call the form $\chi = w^*\lambda \circ j + \sqrt{-1} w^*\lambda$ the \emph{
contact Hick's field} of $w$ and
$$
Q(p) + \sqrt{-1} T(p)
$$
the \emph{asymptotic charge} of the instanton $w$ at the puncture $p$.
\end{defn}

Note that by the closedness $d(w^*\lambda \circ j) = 0$, the charge $Q(p)$ is the same as
the initial integral
$$
\int_{\{\tau = 0\}} w^*\lambda \circ j
$$
which does not depend on the choice of subsequence but is determined by
the initial condition at $\tau = 0$ and homology class of the loop
$w|_{\tau = 0} \in H_1(\dot \Sigma) = H_1(\Sigma \setminus \{p_1, \cdots, p_k\}$.
Furthermore if $u= (w,f)$ is an $\lcs$-instanton, then $Q(p) = 0$.

Now we consider the  asymptotic Hick's field $\chi_\infty$ associated to the asymptotic
$\lcs$ instanton $(w_\infty,\chi_\infty)$ obtained in the proof of Proposition \ref{prop:pole-structure}.
Because $w_\infty$ is massless and has bounded derivatives on $\R \times S^1$,
$\chi_\infty$ becomes a bounded holomorphic one-form. Therefore we derive
\be\label{eq:chiinfty}
\chi_\infty = c\, (d\tau + i \, dt)
\ee
for some complex number $c \in \C$. We denote $c = b + i a$ for $a, \, b \in \R$.
Equivalently, we obtain
$$
w_\infty^*\lambda = a\, d\tau + b\, dt
$$
where $a = - Q(p), \, b = T(p)$ .

\begin{thm} \label{thm:c=0}
Suppose $Q(p) = 0 =T(p)$. Then $u$ is smooth across $p$ and so the
puncture $p$ is removable.
\end{thm}
\begin{proof} When $c = 0$, we obtain $dw_\infty = d^\pi w_\infty
+ \lambda^*w_\infty\, R_\lambda = 0$ and so
$w_\infty$ must be a constant map $q \in Q$. By the convergence $w_j \to w_\infty$ in the $C^\infty$ topology
on every compact subset, it follows that $w_j(0,\cdot) \to q$ or equivalently
$$
d(w|_{r = \delta_j}, q) \to 0
$$
and $w_j^*\lambda \to 0$ converges uniformly. Using the compactness of $Q$ and applying Ascoli-Arzela theorem,
we can choose a sequence $z_i \to p$ in $D_\delta(p) \setminus \{p\}$ such that
$w(z_i) \to p$ and $w^*\lambda|_{r = \delta_j} \to 0$ uniformly. Then this
continuity of $w^*\lambda$ at $p$ in turn implies $dw$ is continuous at $p$
by the expression
$$
dw = d^\pi w  + w^*\lambda\otimes R_\lambda(w)
$$
In particular $|dw|_{D_\delta(r)}$ is bounded and so lies in
$L^2 \cap L^4$ on $D_\delta(r)$.  Then the local $W^{k,2}$ a priori estimate implies
that $w$ is indeed smooth across $p$.

Recalling $f^*d\theta = w^*\lambda \circ j$, smoothness of $f$ across $p$ immediately follows.
This finishes the proof.
\end{proof}

\medskip

If $T \neq 0$, we obtain
$$
\lim_{k \to \infty}\int_{S^1} (w|_{r = \delta_k})^*\lambda  =  T \neq 0.
$$
\begin{prop}\label{prop:bnot0} Suppose $T \neq 0$. Then there exists a closed Reeb orbit $\gamma$ of
period $T = b$ such that there exists a sequence $\tau_k \to \infty$
for which $w(\tau_k, \cdot) \to \gamma(T(\cdot))$ in the $C^\infty$ topology.
\end{prop}
\begin{proof}
When $T \neq 0$, we obtain
$$
dw_\infty =  (T \, dt)\otimes R_\lambda.
$$
Again by the connectedness of $[0,\infty) \times S^1$, it follows that
the image of $w_\infty$ must be contained in a single leaf of the Reeb foliation
and so
$$
w_\infty(\tau,t) = \gamma(T\, t)
$$
for a parameterized Reeb orbit $\gamma$ such that $\dot \gamma = R_\lambda(\gamma)$.
Such a parameterization is unique modulo a time-shift.
Since the map $w$ is one-periodic for any $\tau$, we derive
$$
\gamma(T) = \gamma(0).
$$
This implies first that $\gamma$ is a periodic Reeb orbit of period $T$.
\end{proof}

Now we are ready to define the notion of positive and negative
punctures of contact instanton map $w$. Assume $\lambda$ is nondegenerate.

Let $p$ be one of the punctures of $\dot \Sigma$.
In the disc $D_\delta(p) \subset \C$ with the standard orientation, we consider the function
$$
\int_{\del D_\delta(p)} w^*\lambda
$$
as a function of $\delta > 0$. This function is either decreasing or increasing
by the Stokes' formula, the positivity $w^*d\lambda \geq 0$ and the finiteness of
$\pi$-energy
$$
\frac{1}{2} \int_{\dot \Sigma} |d^\pi w|^2 \, dA = \int_{\dot \Sigma} w^*d\lambda < \infty.
$$
\begin{defn}[Classification of punctures]\label{defn:plus-minus-punctures}
Let $\dot \Sigma$ be a puncture Riemann
surface with punctures $\{p_1, \cdots, p_k\}$ and let
$w: \dot \Sigma \to Q$ be a contact instanton map.
\begin{enumerate}
\item We call a puncture $p$ \emph{removable} if $T(p) = Q(p) = 0$, and \emph{non-removable} otherwise.
\item
We say a non-removable puncture \emph{positive} (resp. \emph{negative}) puncture if the function
$$
\int_{\del D_\delta(p)} w^*\lambda
$$
is increasing (resp. decreasing) as $\delta \to 0$.
\end{enumerate}
\end{defn}
\begin{rem}
The appearance of `adiabatic puncture' (i.e. a puncture with $T=0$ and $Q \neq 0$
is an obstacle towards the compactification of contact instantons in \cite{oh-wang:CR-map1,oh-wang:CR-map2}
\emph{ when the form $w^*\lambda \circ j$ is not exact}. It is quite remarkable
that this obstacle is removed when the contact instanton $w$ is coupled with
a map $f: Q \to S^1$ to form an $\lcs$ instanton $u=(w,f)$.
\end{rem}

\section{Definition of the moduli space of $\lcs$-instantons}

In this section, we define the moduli space of $\lcs$-instantons and
construct its compactification. It turns out that compactification of
$\lcs$-instantons is much simpler since we will not need the well-known SFT-type
compactification that appears in the study of moduli space of pseudoholomorphic
curves in the case of symplectization.

Let $(\dot \Sigma, j)$ be a punctured Riemann surface and let
$$
p_1, \cdots, p_{s^+}, q_1, \cdots, q_{s^-}
$$
be the positive and negative punctures. For each $p_i$ (resp. $q_j$), we associate the isothermal
coordinates $(\tau,t) \in [0,\infty) \times S^1$ (resp. $(\tau,t) \in (-\infty,0] \times S^1$)
on the punctured disc $D_{e^{-2\pi R_0}}(p_i) \setminus \{p_i\}$
(resp. on $D_{e^{-2\pi R_0}}(q_i) \setminus \{q_i\}$) for some sufficiently large $R_0 > 0$.

For each $\eta \in [\dot \Sigma, S^1]$, we define the associated energy
$$
E_\eta(u) = E^\pi(u) + E^\perp_\eta(u)
$$
for each smooth map $u = (w,f)$ in class $\eta$, i.e., $[u]_{S^1} = \eta$.
\begin{defn} For each fixed $\eta$, we define
$$
\widetilde{\CM}_{k,\ell}^\eta(\dot \Sigma, M;J), \quad k, \, \ell \geq 0, \, k+\ell \geq 1
$$
to be the moduli space of $\lcs$-instantons $u = (w,f)$ with $E_\eta(u) < \infty$.
\end{defn}
Then we have a decomposition
$$
\widetilde{\CM}_{k,\ell}^\eta(\dot \Sigma, M;J)
= \bigcup_{\vec \gamma^+,\vec \gamma^-} \widetilde{\CM}_{k,\ell}^\eta
\left(\dot \Sigma, Q;J;(\vec \gamma^+,\vec \gamma^-)\right)
$$
by Theorem \ref{thm:subsequence} where
\beastar
\widetilde{\CM}_{k,\ell}^\eta(\dot \Sigma, Q;J;(\vec \gamma^+,\vec \gamma^-))
& = &\{u = (w,f) \mid u \, \text{is an $\lcs$ instanton with } \\
&{}& \quad E_\eta(u) < \infty, \, w(-\infty_j) = \gamma^-_j, \, w(\infty_i) = \gamma_i \}.
\eeastar
Here we have
the collections of Reeb orbits $\gamma^+_i$ and $\gamma^-_j$
and of points $\theta^+_i$, $\theta^-_j$ for $i =1, \cdots, s^+$
and for $j = 1, \cdots, s^-$ respectively such that
\be\label{eq:limatinfty}
\lim_{\tau \to \infty}w((\tau,t)_i) = \gamma^+_i(T_i(t+t_i)), \quad
\lim_{\tau \to - \infty}w((\tau,t)_j) = \gamma^-_j(T_j(t-t_j))
\ee
for some $t_i, \, t_j \in S^1$, where
$$
T_i = \int_{S^1} (\gamma^+_i)^*\lambda, \, T_j = \int_{S^1} (\gamma^-_j)^*\lambda.
$$
Here $t_i,\, t_j$ depends on the given analytic coordinate and the parameterization of
the Reeb orbits.

Due to the $S^1$-equivariance of the equation \eqref{eq:lcs-instanton}
under the $S^1$ action of rotations on $S^1$, this action induces a free action on
$\widetilde{\CM}_{k,\ell}^\eta(\dot \Sigma, M;J)$. Then we denote
\be\label{eq:CMkell}
\CM_{k,\ell}^\eta(\dot \Sigma, M;J) = \widetilde{\CM}_{k,\ell}(\dot \Sigma, M;J)/S^1.
\ee
We also have the decomposition
$$
\widetilde{\CM}_{k,\ell}^\eta(\dot \Sigma, M;J) = \bigcup_{\vec \gamma^\pm}
\widetilde{\CM}_{k,\ell}^\eta(\dot \Sigma, Q;J;(\vec \gamma^+,\vec \gamma^-)).
$$
Here we denote
$$
\vec \gamma^+ = (\gamma^+_i), \quad \vec \gamma^- = (\gamma^-_j).
$$
The above mentioned $S^1$-action acts on each
$\widetilde{\CM}_{k,\ell}(\dot \Sigma, Q;J;(\vec \gamma^+,\vec \gamma^-))$
by image rotation of the component $f$, and hence
$$
\CM_{k,\ell}^\eta(\dot \Sigma, M;J) = \bigcup_{\vec \gamma^\pm}
\CM_{k,\ell}^\eta(\dot \Sigma, Q;J;(\vec \gamma^+,\vec \gamma^-)).
$$
For each element $(w,f)$ of $\widetilde{\CM}_{k,\ell}^\eta(\dot \Sigma, Q;J;(\vec \gamma^+,\vec \gamma^-))$,
we have
points $(\vec \theta^+, \vec \theta^-) \in (S^1)^k \times (S^1)^\ell$ such that
$$
\lim_{\tau \to \infty}f((\tau,\cdot)_i )= \theta^+_i,  \quad \lim_{\tau \to \infty}f((\tau,\cdot)_j )= \theta^-_j.
$$
We now introduce a uniform energy bound for $u$ with given asymptotic condition at
its punctures. To describe this energy estimate, we
consider the associated function $\widetilde f: \dot \Sigma \to \R$
determined by
\be\label{eq:tildef}
f^*d\theta - \beta_\eta = d\widetilde f
\ee
for given $u = (w,f) \in \CM_{k,\ell}^\eta(\dot \Sigma, Q;J;(\vec \gamma^+,\vec \gamma^-))$.

The following proposition is the analog to \cite[Lemma 5.15]{behwz} and
\cite[Proposition 9.2]{oh:contacton}  whose proof is also similar.

\begin{prop}\label{prop:proper-energy} Let $\eta \in [\dot \Sigma,S^1]$ and
$u = (w,f) \in \CM_{k,\ell}^\eta(\dot \Sigma, Q;J;(\vec \gamma^+,\vec \gamma^-))$.
Suppose that $E^\pi(w) < \infty$ and the function $f:\dot \Sigma \to \R$ is proper.
Then $E(w) < \infty$.
\end{prop}
\begin{proof} We postpone the details of the proof till Appendix \ref{sec:energy-bound}.
\end{proof}

By the same argument as the derivation of \cite[Lemma 5.16]{behwz}, we obtain
the following a priori energy bounds for the $\text{\rm lcs}$ instantons $u = (w,f)$ in a \emph{fixed}
charge class $[u] \in H^1(\dot \Sigma)$.

\begin{prop}\label{prop:uniform-energy-bound} Let $\eta \in [\dot \Sigma,S^1]$ and
$u = (w,f) \in \CM_{k,\ell}^\eta(\dot \Sigma, Q;J;(\vec \gamma^+,\vec \gamma^-))$.
Let $\widetilde f: \dot \Sigma \to \R$ be the function given as \eqref{eq:tildef}.
Suppose the function $\widetilde f$ is proper. Then we have
\beastar
E^\pi(u) & = &  E^\pi(w) = \sum_{i=1}^k \CA_\lambda(\gamma^+_i) - \sum_{j=1}^\ell \CA_\lambda(\gamma^-_j)\\
E^\perp_\eta (u)& = & \sum_{j=1}^k \CA_\lambda(\gamma^+_i)  \\
E(u) & = & 2  \sum_{i=1}^k \CA_\lambda(\gamma^+_i) - \sum_{j=1}^\ell \CA_\lambda(\gamma^-_j).
\eeastar
\end{prop}

\section{Fredholm theory and index calculations}
\label{sec:Fredholm}

In this section, we work out the Fredholm theories of $\lcs$-instantons. We will adapt the
exposition given in \cite{oh:contacton} for the case of contact instantons to that of $\lcs$-instantons
by incorporating the presence of the $S^1$-factor in the product $M = Q^{2n-1} \times S^1$.

We divide our discussion into the closed case and the punctured case.

\subsection{Calculation of the linearization map}
\label{subsec:linearization}

Let $\Sigma$ be a closed Riemann surface and $\dot \Sigma$ be its
associated punctured Riemann surface.  We allow the set of whose punctures
to be empty, i.e., $\dot \Sigma = \Sigma$.
We would like to regard the assignment $u \mapsto \delbar_J u$ which can be
decomposed into
$$
u= (w,f) \mapsto \left(\delbar^\pi w, w^*\lambda \circ j - f^*d\theta\right)
$$
for a map $w: \dot \Sigma \to Q$ as a section of the (infinite dimensional) vector bundle
over the space of maps of $w$. In this section, we lay out the precise relevant off-shell framework
of functional analysis.

Let $(\dot \Sigma, j)$ be a punctured Riemann surface, the set of whose punctures
may be empty, i.e., $\dot \Sigma = \Sigma$ is either a closed or a punctured
Riemann surface. We will fix $j$ and its associated K\"ahler metric $h$.

We consider the map
$$
\Upsilon(w,f) = \left(\delbar^\pi w, w^*\lambda \circ j - f^*d\theta \right)
$$
which defines a section of the vector bundle
$$
\CH \to \CF = C^\infty(\Sigma,Q)
$$
whose fiber at $u \in C^\infty(\Sigma,Q \times S^1)$ is given by
$$
\CH_u: = \Omega^{(0,1)}(u^*\xi) \oplus \Omega^{(0,1)}(u^*\CV).
$$
Recalling $\CV_{(q,\theta)} = \span_\R\{R_\lambda, \frac{\del}{\del \theta}\}$,
we have a natural isomorphism
$$
\Omega^{(0,1)}(u^*\CV) \cong \Omega^1(\Sigma) = \Omega^0(T^*S^1)
$$
via the map
$$
\alpha \in \Omega^0(T^*S^1) \mapsto \frac12\left(\alpha \frac{\del}{\del \theta} + \alpha \circ j R_\lambda\right).
$$
Utilizing this isomorphism,
we decompose $\Upsilon = (\Upsilon_1,\Upsilon_2)$ where
\be\label{eq:upsilon1}
\Upsilon_1: \Omega^0(w^*TQ) \to \Omega^{(0,1)}(w^*\xi); \quad \Upsilon_1(w) = \delbar^\pi(w)
\ee
and
\be\label{eq:upsilon2}
\Upsilon_2: \Omega^0(w^*TQ) \to \Omega^1(\Sigma); \quad \Upsilon_2(w) = w^*\lambda \circ j - f^*d\theta.
\ee

We first compute the linearization map which defines a linear map
$$
D\Upsilon(u) : \Omega^0(w^*TQ) \oplus \Omega^0(f^*TS^1) \to \Omega^{(0,1)}(w^*\xi) \oplus \Omega^1(\dot \Sigma)
$$
where we have
$$
T_u \CF = \Omega^0(w^*TQ) \oplus \Omega^0(f^*TS^1).
$$

For the optimal expression of the linearization map and its relevant
calculations, we use the $\mathfrak{lcs}$-fication connection $\nabla$ of $(Q \times S^1,\lambda,J)$
which is the lcs-lifting of contact triad connection introduced in \cite{oh-wang:CR-map1}.
We refer readers to \cite{oh-wang:CR-map1}, \cite{oh:contacton} for the unexplained notations
appearing in our tensor calculations during the proof.

\begin{thm}\label{thm:linearization} We decompose $d\pi = d^\pi w + w^*\lambda\otimes R_\lambda$
and $Y = Y^\pi + \lambda(Y) R_\lambda$, and $X = (Y, v) \in \Omega^0(u^*T(Q \times S^1))$.
Denote $\kappa = \lambda(Y)$ and $\upsilon = d\theta(v)$. Then we have
\bea
D\Upsilon_1(u)(Y,v) & = & \delbar^{\nabla^\pi}Y^\pi + B^{(0,1)}(Y^\pi) +  T^{\pi,(0,1)}_{dw}(Y^\pi) \nonumber\\
&{}& \quad + \frac{1}{2}\kappa \cdot  \left((\CL_{R_\lambda}J)J(\del^\pi w)\right)
\label{eq:Dwdelbarpi}\\
D\Upsilon_2(u)(Y,v) & = &  (\CL_Y \lambda) \circ j- \CL_v d\theta = d\kappa \circ j - d\upsilon
+ Y \rfloor d\lambda \circ j
\nonumber\\
\label{eq:Dwddot}
\eea
where $B^{(0,1)}$ and $T_{dw}^{\pi,(0,1)}$ are the $(0,1)$-components of $B$ and
$T_{dw}^{\pi,(0,1)}$, where $B, \, T_{dw}^\pi: \Omega^0(w^*TQ) \to \Omega^1(w^*\xi)$ are
zero-order differential operators given by
$$
B(Y) =
- \frac{1}{2}  w^*\lambda \left((\CL_{R_\lambda}J)J Y\right)
$$
and
$$
T_{dw}^\pi(Y) = \pi T(Y,dw)
$$
respectively.
\end{thm}
\begin{proof}
Let $Y$ be a vector field over $w$ and $w_s$ be a family of maps $w_s: \Sigma \to Q$ with
$w_0 = w$ and $Y = \frac{d}{ds}\Big|_{s = 0} w^s$, and $a = \frac{d\gamma}{dt}\Big|_{t = 0}$
for a curve $\gamma$ with $\gamma(0) = z$. We decompose
$$
Y = Y^\pi + \kappa \, R_\lambda, \quad \kappa = \lambda(Y)
$$
into the sum of $\xi$-component and $R_\lambda$-component.

Calculation of \eqref{eq:Dwdelbarpi} is the same as that of \cite[Theorem 10.1]{oh:contacton} which
we refer readers to for the details.

Next we compute $D\Upsilon_2(u)$ and prove \eqref{eq:Dwddot}.
We compute $\frac{d}{ds}|_{s= 0} w_s^*\lambda \circ j$
\be\label{eq:nablas}
\frac{d}{ds}\Big|_{s = 0} w_s^*\lambda\circ j
= \frac{d}{ds}\Big|_{s = 0}w_s^*\lambda \circ j.
\ee
By Cartan's formula applied to the \emph{vector field $Y$ over the map $w$},
we obtain
$$
\frac{d}{ds}\Big|_{s = 0} w_s^*\lambda = Y \rfloor d\lambda + d(Y \rfloor \lambda)
$$
where $\rfloor$ is the interior product over the map $w$.
Substituting this into \eqref{eq:nablas}, we derive
$$
\frac{d}{ds}\Big|_{s = 0} w_s^*\lambda\circ j  =
d(\lambda(Y))\circ j + (Y \rfloor d\lambda) \circ j = d\kappa \circ j + (Y \rfloor d\lambda) \circ j.
$$
Similarly we compute
$$
\frac{d}{ds}\Big|_{s = 0} f_s^*d\theta = d(v \rfloor d\theta).
$$
This proves
\be\label{eq:linearizedUpsilon2}
D\Upsilon_2(w)(Y) = d\kappa \circ j + (Y \rfloor d\lambda) \circ j - d\upsilon, \, \upsilon = d\theta(v)
\ee
which finishes the proof of Theorem \ref{thm:linearization}.
\end{proof}

Now we evaluate $D\Upsilon_1(w)$
more explicitly. We have
$$
\delbar^{\nabla^\pi}Y = \frac{1}{2} \left(\nabla^\pi Y + J \nabla^\pi_{j(\cdot)} Y \right)
$$
and $B^{(0,1)}(Y)$ becomes
$$
-\frac{1}{4}\left(w^*\lambda\,  \pi ((\CL_{R_\lambda}J)J Y) + w^*\lambda \circ j\, \pi (\CL_{R_\lambda}J)Y \right).
$$

\subsection{The closed case}

We recall the classification result of contact instantons, Proposition \ref{prop:abbas}.
This (or rather its proof given in \cite[Proposition 3.4]{oh-wang:CR-map1}) gives rise to the following
immediate classification result
for smooth $\lcs$ instantons too.

\begin{prop}\label{prop:classify} Let $u=(w,f) : \Sigma \to Q \times S^1$ be a smooth $\lcs$ instanton from
a closed Riemann surface $(\Sigma,j)$. Then
\begin{enumerate}
\item If $g(\Sigma) = 0$, then $u$ is a constant map.
\item If $g(\Sigma) \geq 1$, there exists a closed Reeb orbit $C$ such that
$u$ can be factored into
$u = \iota_u \circ \varphi_u$ where $\varphi_u: \Sigma \to S^1 \times C$
is a surjective map and $\iota_u: C \times S^1 \to Q \times S^1$ is the inclusion map.
\end{enumerate}
Furthermore if we equip $C \times S^1$ with the complex structure induced from $J$ and denote
by $E_{C,J}$ the resulting elliptic curve, then the map $\varphi_u: \Sigma \to E_{C,J}$
is a holomorphic branched covering. When $g(\Sigma) =1 $, $\varphi_u$ is an \'etale covering.
\end{prop}
\begin{proof} We recall that $w$ is a contact instanton.
The beginning of our proof closely follows that of \cite[Proposition 3.4]{oh-wang:CR-map1}
but for readers' convenience, we duplicate it here.
Since $\Sigma$ is closed,  the identity $|d^\pi w|^2 dA = d(2w^*\lambda)$ implies
$d^\pi w = 0$ by Stokes' formula. Therefore we have
$$
dw = w^*\lambda \otimes R_\lambda.
$$
In particular the image of $w$ is contained in a single leaf, say $C$, of the Reeb foliation of $(Q,\lambda)$.
It was proven in the proof of \cite[Proposition 3.4]{oh-wang:CR-map1} that if $C$ is not a closed leaf,
$w$ must be a constant map and so is $u$. Therefore we have only to consider the case where $C$ is a
closed leaf.

The vanishing $d^\pi w= 0$ also implies $dw^*\lambda = 0$. Combined with $d(w^*\lambda \circ j) = 0$,
we conclude that $w^*\lambda$ is a harmonic one-form on $\Sigma$. Then we derive that
the one-form
$$
\alpha = f^*d\theta + \sqrt{-1} w^*\lambda = w^*\lambda \circ j + \sqrt{-1} w^*\lambda
$$
is a smooth holomorphic one-form on $\Sigma$. Recall
$$
\dim_\C H^{(1,0)}(\Sigma) = g.
$$
Therefore if $g(\Sigma) = 0$, it must be zero and so $u$ is a constant map.

When $g(\Sigma) \geq 1$, we note that
the map $u$ is factored through $\Sigma \to C \times S^1 \hookrightarrow Q \times S^1$.
We note that $C \times S^1$ is diffeomorphic to a two-torus. Furthermore it is equipped with
the induced complex structure satisfying
$$
J\frac{\del}{\del \theta} = R_\lambda = \dot \gamma, \quad J \dot \gamma = - \frac{\del}{\del \theta}.
$$
We denote the resulting elliptic curve by $(E_{C,J},J)$ by some abuse of notation.
Obviously the map $u$ induces a (non-constant)
holomorphic map $\varphi_u: (\Sigma,j) \to (E_{C,J},J)$ since the map $u$ is $(j,J)$-holomorphic.
 By the general property of holomorphic maps between two
closed Riemann surfaces, it follows that $\varphi_u$ is a branched covering map
onto $E_{C,J}$, provided it is not constant. Furthermore if $g(\Sigma) = 1$, it is an \'etale covering, i.e.,
a covering map without branching. In particular if the degree is 1, the map is an
isomorphism of an elliptic curve. This finishes the proof.
\end{proof}
\begin{rem}
This proposition has two implications. On the one hand, there are `many' closed $\lcs$ instantons
for the genus zero case which are all constants, and existence of nonconstant closed $\lcs$ instanton for
the higher genus case implies existence of a closed Reeb orbit of $(Q,\lambda)$.
Therefore if one wants to develop a more interesting story of contact instantons, one
should consider those defined on open Riemann surfaces, either with boundary or with punctures.
\end{rem}

Next we would like to compute the Fredholm index of such an $\lcs$ instanton
$u = (f, w)$. We start with considering the $w$ component.

From the expression of the map $\Upsilon = (\Upsilon_1,\Upsilon_2)$, the map defines a bounded linear map
\be\label{eq:dUpsilon}
D\Upsilon(w): \Omega^0_{k,p}(w^*TQ) \to \Omega^{(0,1)}_{k-1,p}(w^*\xi) \oplus \Omega^2_{k-2,p}(\Sigma).
\ee
We choose $k \geq 2, \, p > 2$. Recalling the decomposition
$$
Y = Y^\pi + \lambda(Y)\, R_\lambda,
$$
we have the decomposition
$$
\Omega^0_{k,p}(w^*TQ) \cong \Omega^0_{k,p}(w^*\xi) \oplus \Omega^0_{k,p}(\dot \Sigma,\R)\cdot R_\lambda.
$$
Here we use the splitting
$$
TQ = \span_\R\{R_\lambda\} \oplus \xi
$$
where $\span_\R\{R_\lambda\}: = \CL$ is a trivial line bundle and so
$$
\Gamma(w^*\CL) \cong C^\infty(\Sigma).
$$
By definition the linearization operator $D\Upsilon_2(w)$ acts trivially on the section
$Y$ tangent to the Reeb direction.

We derive the following index formula for $D\Upsilon(w)$ from the homotopy
invariance of the index

\begin{thm}\label{thm:index} Let $\Sigma$ be any closed Riemann surface
of genus $g$, and let $w: \Sigma \to Q$ be a contact instanton
with finite $\pi$-energy. Then the operator \eqref{eq:dUpsilon} is a Fredholm operator
whose index is given by
\be\label{eq:indexwhen0}
\operatorname{Index} D\Upsilon(w) =  2n(1-g).
\ee
\end{thm}
\begin{proof}

We already know that the operators $\delbar^{\nabla^\pi} +  T_{dw}^{\pi,(0,1)} + B^{(0,1)}$ and
$\delbar$ are Fredholm. Furthermore we can homotope the operator
\eqref{eq:dUpsilon} to the direct sum operator
$$
\left(\delbar^{\nabla^\pi}  + T_{dw}^{\pi,(0,1)} + B^{(0,1)} + \frac{1}{2}
\lambda(\cdot) (\CL_{R_\lambda}J)J \del^\pi w \right)
\oplus \delbar
$$
 Then the index is given by
\beastar
\operatorname{Index}\delbar^{\nabla^\pi} + \operatorname{Index} (\delbar)
& = & 2c_1(w^*\xi) + 2(n-1)(1-g) + 2(1-g)\\
& = & 2c_1(w^*\xi) + 2n(1-g).
\eeastar
in general. But since  $[w] = 0$ in $H_2(Q;\Z)$ for any contact instanton
by Proposition \ref{prop:abbas}, this is reduced to
\eqref{eq:indexwhen0}. This finishes the proof.
\end{proof}

\subsection{The punctured case}
\label{subsec:punctured}

We start with the general remark on the Fredholm theory of
nonlinear elliptic equation defined on the non-compact domain. The
Fredholm theory usually requires both the domain and the linearized
operator to be \emph{cylindrical} in that the operator is translationally
invariant under the radial coordinates. For the current $\lcs$ instanton
equation, the linearized operator is not cylindrical when the
asymptotic charge $Q(p)$ is nonzero and so the linearization operator
derived in Section \ref{subsec:linearization} is not suitable for the
Fredholm theory. Here enters the charge class $[u]_{S^1}$ again!

By restricting the charge class $[u]_{S^1}$, we restrict to the
smaller off-shell function space
$$
\CW^{k,p}_{\delta;\eta}\left(\dot \Sigma,Q \times S^1;\vec \gamma^+, \vec \gamma^-\right)
$$
and linearize the map
$$
(\widetilde f, w) \mapsto \left(\delbar^\pi w, \beta_\eta + d\widetilde f\right)
$$
with the harmonic part $\beta_\eta$ \emph{fixed} instead. This restricted linearization
operator then becomes cylindrical. With this preparation, the rest of the discussion
will be restricted to this smaller function space and so the linearized
operator becomes cylindrical.  Then the Fredholm theory details
are entirely similar to the case of a symplectization. For readers' convenience, we
provide full details of the Fredholm theory and the index calculation.

Fix an elongation function $\rho: \R \to [0,1]$
so that
\beastar
\rho(\tau) & = & \begin{cases} 1 \quad & \tau \geq 1 \\
0 \quad & \tau \leq 0
\end{cases} \\
0 & \leq & \rho'(\tau) \leq 2.
\eeastar

Then we consider sections of $w^*TQ$ by
\be\label{eq:barY}
\overline Y_i = \rho(\tau - R_0) R_\lambda(\gamma^+_k(t)),\quad
\underline Y_j = \rho(\tau + R_0) R_\lambda(\gamma^+_k(t))
\ee
and denote by $\Gamma_{s^+,s^-} \subset \Gamma(w^*TQ)$ the subspace defined by
$$
\Gamma_{s^+,s^-} = \bigoplus_{i=1}^{s^+} \R\{\overline Y_i\} \oplus \bigoplus_{j=1}^{s^-} \R\{\underline Y_j\}.
$$
Let $k \geq 2$ and $p > 2$. We denote by
$$
\CW^{k,p}_\delta(\dot \Sigma, Q;J;\gamma^+,\gamma^-), \quad k \geq 2
$$
the Banach manifold such that
\be\label{eq:limatinfty}
\lim_{\tau \to \infty}w((\tau,t)_i) = \gamma^+_i(T_i(t+t_i)), \quad
\lim_{\tau \to - \infty}w((\tau,t)_j) = \gamma^-_j(T_j(t-t_j))
\ee
for some $t_i, \, t_j \in S^1$, where
$$
T_i = \int_{S^1} (\gamma^+_i)^*\lambda, \, T_j = \int_{S^1} (\gamma^-_j)^*\lambda.
$$
Here $t_i,\, t_j$ depends on the given analytic coordinate and the parameterization of
the Reeb orbits.

The local model of the tangent space  of $\CW^{k,p}_\delta(\dot \Sigma, Q;J;\gamma^+,\gamma^-)$ at
$w \in C^\infty_\delta(\dot \Sigma,Q) \subset W^{k,p}_\delta(\dot \Sigma, Q)$ is given by
\be\label{eq:tangentspace}
\Gamma_{s^+,s^-} \oplus W^{k,p}_\delta(w^*TQ)
\ee
where $W^{k,p}_\delta(w^*TQ)$ is the Banach space
\beastar
&{} & \{Y = (Y^\pi, \lambda(Y)\, R_\lambda)
\mid e^{\frac{\delta}{p}|\tau|}Y^\pi \in W^{k,p}(\dot\Sigma, w^*\xi), \,
\lambda(Y) \in W^{k,p}(\dot \Sigma, \R)\}\\
& \cong & W^{k,p}(\dot \Sigma, \R) \cdot R_\lambda(w) \oplus W^{k,p}(\dot\Sigma, w^*\xi).
\eeastar
Here we measure the various norms in terms of the triad metric of the triad $(Q,\lambda,J)$.
To describe the choice of $\delta > 0$, we need to recall the covariant linearization of the map $
D\Upsilon_{\lambda, T}: W^{1,2}(z^*\xi) \to L^2(z^*\xi)
$
of the map
$$
\Upsilon_{\lambda,T}: z \mapsto \dot z - T\, R_\lambda(z)
$$
for a given $T$-periodic Reeb orbit $(T,z)$. The operator has the expression
\be\label{eq:DUpsilon}
D\Upsilon_{\lambda, T} = \frac{D^\pi}{dt} - \frac{T}{2}(\CL_{R_\lambda}J) J=: A_{(T,z)}
\ee
where $\frac{D^\pi}{dt}$ is the covariant derivative
with respect to the pull-back connection $z^*\nabla^\pi$ along the Reeb orbit
$z$ and $(\CL_{R_\lambda}J) J$ is a (pointwise) symmetric operator with respect to
the triad metric. (See Lemma 3.4 \cite{oh-wang:connection}.)
We choose $\delta> 0$ so that $0 < \delta/p < 1$ is smaller than the
spectral gap
\be\label{eq:gap}
\text{gap}(\gamma^+,\gamma^-): = \min_{i,j}
\{d_{\text H}(\text{spec}A_{(T_i,z_i)},0),\, d_{\text H}(\text{spec}A_{(T_j,z_j)},0)\}.
\ee

Now for each given $w \in \CW^{k,p}_\delta:= \CW^{k,p}_\delta(\dot \Sigma, Q;J;\gamma^+,\gamma^-)$,
we consider the Banach space
$$
\Omega^{(0,1)}_{k-1,p;\delta}(w^*\xi)
$$
the $W^{k-1,p}_\delta$-completion of $\Omega^{(0,1)}(w^*\xi)$ and form the bundle
$$
\CH^{(0,1)}_{k-1,p;\delta}(\xi) = \bigcup_{w \in \CW^{k,p}_\delta} \Omega^{(0,1)}_{k-1,p;\delta}(w^*\xi)
$$
over $\CW^{k,p}_\delta$. Then we can regard the assignment
$$
\Upsilon_1: (w,f) \mapsto \delbar^\pi w
$$
as a smooth section of the bundle $\CH^{(0,1)}_{k-1,p;\delta}(\xi) \to \CW^{k,p}_\delta$.

Furthermore the assignment
$$
\Upsilon_2: (w,f) \mapsto w^*\lambda \circ j - f^*d\theta
$$
defines a smooth section of the bundle
$$
\Omega^1_{k-1,p}(u^*\CV) \to \CW^{k,p}_\delta.
$$
We have already computed the linearization of each of these maps in the previous section.

With these preparations, the following is a corollary of exponential estimates established
in \cite{oh-wang:CR-map1}.

\begin{prop}[Corollary 6.5 \cite{oh-wang:CR-map1}]\label{prop:on-containedin-off}
Assume $\lambda$ is nondegenerate.
Let $w:\dot \Sigma \to Q$ be a contact instanton and let $w^*\lambda = a_1\, d\tau + a_2\, dt$.
Suppose
\bea
\lim_{\tau \to \infty} a_{1,i} = -Q(p_i), &{}& \, \lim_{\tau \to \infty} a_{2,i} = T(p_i)\nonumber\\
\lim_{\tau \to -\infty} a_{1,j} = -Q(p_j) , &{}& \, \lim_{\tau \to -\infty} a_{2,j} = T(p_j)
\eea
at each puncture $p_i$ and $q_j$.
Then $w \in \CW^{k,p}_\delta(\dot \Sigma, Q;J;\gamma^+,\gamma^-)$.
\end{prop}

Now we are ready to describe the moduli space of $\lcs$ instantons with prescribed
asymptotic condition as the zero set
\be\label{eq:defn-MM}
\CM(\dot \Sigma, Q;J;\gamma^+,\gamma^-) = \left(\CW^{k,p}_\delta(\dot \Sigma, Q;J;\gamma^+,\gamma^-)
\oplus \CW^{k,p}_\delta(\dot \Sigma, S^1)\right) \cap \Upsilon^{-1}(0)
\ee
whose definition does not depend on the choice of $k, \, p$ or $\delta$ as long as $k\geq 2, \, p>2$ and
$\delta > 0$ is sufficiently small. One can also vary $\lambda$ and $J$ and define the universal
moduli space whose detailed discussion is postponed.

In the rest of this section, we establish the Fredholm property of
the linearization map
$$
D\Upsilon_{(\lambda,T)}(u): \Omega^0_{k,p;\delta}(u^*T(Q \times S^1);J;\gamma^+,\gamma^-) \to
\Omega^{(0,1)}_{k-1,p;\delta}(w^*\xi) \oplus \Omega^{(0,1)}_{k-1,p}(u^*TS^1)
$$
and compute its index. Here we also denote
$$
\Omega^0_{k-1,p;\delta}(u^*T(Q \times S^1);J;\gamma^+,\gamma^-) =
W^{k-1,p}_\delta(u^*T(Q \times S^1);J;\gamma^+,\gamma^-)
$$
for the semantic reason.

For this purpose, we remark that
as long as the set of punctures is non-empty, the symplectic vector bundle
$w^*\xi \to \dot \Sigma$ is trivial. We denote by
$
\Phi: E \to \overline \Sigma \times \R^{2n}
$
and by
$$
\Phi_i^+: = \Phi|_{\del_i^+ \overline \Sigma}, \quad \Phi_j^- = \Phi|_{\del_j^- \overline \Sigma}
$$
its restrictions on the corresponding boundary components of $\del \overline \Sigma$.
Using the cylindrical structure near the punctures,
we can extend the bundle to the bundle $E \to \overline \Sigma$ where $\overline \Sigma$
is the real blow-up of the punctured Riemann surface $\dot \Sigma$.

We then consider the following set
\beastar
\CS &: = & \{A: [0,1] \to Sp(2n,\R) \mid 1 \not \in \text{spec}(A(1)), \\
&{}& \hskip0.5in A(0) = id, \, \dot A(0) A(0)^{-1} = \dot A(1) A(1)^{-1} \}
\eeastar
of regular paths in $Sp(2n,\R)$ and denote by $\mu_{CZ}(A)$ the Conley-Zehnder index of
the paths following \cite{robbin-salamon}. Recall that for each closed Reeb orbit $\gamma$ with a fixed
trivialization of $\xi$, the covariant linearization $A_{(T,z)}$ of the Reeb flow along $\gamma$
determines an element $A_\gamma \in \CS$. We denote by $\Psi_i^+$ and $\Psi_j^-$
the corresponding paths induced from the trivializations $\Phi_i^+$ and $\Phi_j^-$ respectively.

We have the decomposition
$$
\Omega^0_{k,p;\delta}(w^*T(Q \times S^1);J;\gamma^+,\gamma^-) =
\Omega^0_{k,p;\delta}(w^*\xi) \oplus \Omega^0_{k,p;\delta}(u^*\CV)
$$
and again the operator
\be\label{eq:DUpsilonu}
D\Upsilon_{(\lambda,T)}(u): \Omega^0_{k,p;\delta}(w^*T(Q \times S^1);J;\gamma^+,\gamma^-) \to
\Omega^{(0,1)}_{k-1,p;\delta}(w^*\xi) \oplus \Omega^{(0,1)}_{k-1,p;\delta}(u^*\CV)
\ee
which is decomposed into
$$
D\Upsilon_1(u)(Y,v)\oplus D\Upsilon_2(u)
$$
where the summands are given as in
\eqref{eq:Dwdelbarpi} and \eqref{eq:Dwddot} respectively. We see therefrom that
$D\Upsilon_{(\lambda,T)}$ is the first-order differential operator whose first-order part
is given by the direct sum operator
$$
(Y^\pi,(\kappa, \upsilon)) \mapsto \delbar^{\nabla^\pi} Y^\pi \oplus (d\kappa \circ j - d\upsilon)
$$
where we write $(Y,v) = \left(Y^\pi + \kappa R_\lambda, \upsilon \frac{\del}{\del \theta}\right)$
for $\kappa = \lambda(Y), \, \upsilon = d\theta(v)$.
Here we have
$$
\delbar^{\nabla^\pi} : \Omega^0_{k,p;\delta}(w^*\xi;J;\gamma^+,\gamma^-) \to
\Omega^{(0,1)}_{k-1,p;\delta}(w^*\xi)
$$
and the second summand can be written as the standard Cauchy-Riemann operator
$$
\delbar: W^{k,p}(\dot \Sigma;\C) \to \Omega^{(0,1)}_{k-1,p}(\dot \Sigma,\C); \quad
\upsilon + i \kappa =: \varphi \mapsto
\delbar \varphi.
$$

The following proposition can be derived from the arguments used by
Lockhart and McOwen \cite{lockhart-mcowen}. However before applying their
general theory, one needs to pay some preliminary
measure to handle the fact that the order of the operators $D\Upsilon(w)$ are
different depending on the direction of $\xi$ or on that of $R_\lambda$.

\begin{prop}\label{prop:fredholm} Suppose $\delta > 0$ satisfies the inequality
$$
0< \delta < \min\left\{\frac{\text{\rm gap}(\gamma^+,\gamma^-)}{p}, \frac{2}{p}\right\}
$$
where $\text{\rm gap}(\gamma^+,\gamma^-)$ is the spectral gap, given in \eqref{eq:gap},
of the asymptotic operators $A_{(T_j,z_j)}$ or $A_{(T_i,z_i)}$
associated to the corresponding punctures. Then the operator
\eqref{eq:DUpsilonu} is Fredholm.
\end{prop}
\begin{proof} We first note that the operators $\delbar^{\nabla^\pi} + T^{\pi,(0,1)}_{dw}  + B^{(0,1)}$ and
$\delbar$ are Fredholm: The relevant a priori coercive $W^{k,2}$-estimates for any integer $k \geq 1$
for the derivative $dw$ on the punctured Riemann surface $\dot \Sigma$ with cylindrical metric
near the punctures are established in \cite{oh-wang:CR-map1} for the operator
$\delbar^{\nabla^\pi}  + T^{\pi,(0,1)}_{dw} + B^{(0,1)}$ and the one for $\delbar$ is standard.
From this, the standard interpolation inequality establishes the $W^{k,p}$-estimates
for $D\Upsilon(w)$ for all $k \geq 2$ and $p \geq 2$.

Secondly, we can express the operator $D\Upsilon(u)$ in a matrix form
\be\label{eq:matrix-form}
\left(
\begin{matrix}
 \delbar^{\nabla^\pi} + B^{(0,1)} +  T^{\pi,(0,1)}_{dw} &, & \frac{1}{2}(\cdot) \cdot
  \left((\CL_{R_\lambda}J)J(\del^\pi w)\right)\\
\left((\cdot)^\pi \rfloor d\lambda\right)\circ j &, & \delbar
\end{matrix}
\right)
\ee
with respect to the decomposition
$$
\left(Y, \upsilon \frac{\del}{\del \theta}\right) = \left(Y^\pi+\kappa R_\lambda, \upsilon\frac{\del}{\del \theta}\right)
\cong (Y^\pi, \upsilon + i\kappa)
$$
in terms of the splitting
$$
T(Q\times S^1) = \xi \oplus (\span\{R_\lambda\} \oplus TS^1) \cong \xi \oplus \C.
$$
Therefore it can be homotoped to the block-diagonal form, i.e., into the direct sum operator
$$
\left(\delbar^{\nabla^\pi} + T^{\pi,(0,1)}_{dw}  + B^{(0,1)}\right)\oplus \delbar
$$
via a continuous path of Fredholm operators given by
$$
s \in [0,1] \mapsto \left(
\begin{matrix}
 \delbar^{\nabla^\pi} + B^{(0,1)} +  T^{\pi,(0,1)}_{dw} &, &\frac{s}{2}(\cdot) \cdot
 \left((\CL_{R_\lambda}J)J(\del^\pi w)\right)\\
s \left((\cdot)^\pi \rfloor d\lambda\right)\circ j &, &\delbar
\end{matrix}
\right)
$$
from $s =1$ to $s = 0$. The Fredholm property of this path follows from the fact that
the off-diagonal terms are $0$-th order linear operators.
\end{proof}

Then by the continuous invariance of the Fredholm index, we obtain
\be\label{eq:indexDXiw}
\operatorname{Index} D\Upsilon_{(\lambda,T)}(w) =
\operatorname{Index} \left(\delbar^{\nabla^\pi} + T^{\pi,(0,1)}_{dw}  + B^{(0,1)}\right)
 + \operatorname{Index}(\delbar).
\ee
Therefore it remains to compute the latter two indices.

We denote by $m(\gamma)$ the multiplicity of the Reeb orbit in general.
Then we have the following index formula.

\begin{thm}\label{thm:indexforDUpsilon} We fix a trivialization
$\Phi: E \to \overline \Sigma$ and denote
by $\Psi_i^+$ (resp. $\Psi_j^-$) the induced symplectic paths associated to the trivializations
$\Phi_i^+$ (resp. $\Phi_j^-$) along the Reeb orbits $\gamma^+_i$ (resp. $\gamma^-_j$) at the punctures
$p_i$ (resp. $q_j$) respectively. Then we have
\bea
&{}& \operatorname{Index} (\delbar^{\nabla^\pi} + T^{\pi,(0,1)}_{dw}  + B^{(0,1)}) \nonumber\\
& = &
n(2-2g-s^+ - s^-) + 2c_1(w^*\xi)  + (s^+ + s^-) \nonumber \\
&{}& \quad  + \sum_{i=1}^{s^+} \mu_{CZ}(\Psi^+_i)- \sum_{j=1}^{s^-} \mu_{CZ} (\Psi^-_j)
\label{eq:Indexdelbarpi}
\eea
\be
\operatorname{Index} (\delbar) = 2\sum_{i=1}^{s^+} m(\gamma^+_i)+ 2\sum_{j=1}^{s^-} m(\gamma^-_j) -2 g.
\label{eq:indexdelbar}
\ee
In particular,
\bea\label{eq:indexforDUpsilon}
&{}& \operatorname{Index} D\Upsilon_{(\lambda,T)}(u) \nonumber\\
& = & n(2-2g-s^+ - s^-) + 2c_1(w^*\xi)\nonumber\\
&{}&  + \sum_{i=1}^{s^+} \mu_{CZ}(\Psi^+_i)
- \sum_{j=1}^{s^-} \mu_{CZ}(\Psi^-_j)\nonumber \\
&{}&  +
\sum_{i=1}^{s^+} (2m(\gamma^+_i)+1) + \sum_{j=1}^{s^-}( 2m(\gamma^-_j)+1)  - 2g.
\eea
\end{thm}
\begin{proof} The formula \eqref{eq:Indexdelbarpi} can be immediately derived from
the general formula given in the top of p. 52 of Bourgeois's thesis \cite{bourgeois}:
The summand $(s^+ + s^-)$ comes from the factor $\Gamma_{s^+,s^-}$ in the decomposition
\eqref{eq:tangentspace} which has dimension $s^+ + s^-$.

So it remains to compute the index \eqref{eq:indexdelbar}.
To compute the (real) index of $\delbar$, we consider the Dolbeault complex
$$
0 \to \Omega^0(\Sigma; D) \to \Omega^1(\Sigma;D) \to 0
$$
where $D = D^+ + D^-$ is the divisor associated to the set of punctures
$$
D^+ =  \sum_{i=1}^{s^+}m(\gamma^+_i) p_i, \quad
D^- = \sum_{j=1}^{s^-} m(\gamma^-_j) q_j
$$
where $m(\gamma^+_i)$ (resp. $m(\gamma^-_j)$) is the multiplicity of the Reeb orbit
$\gamma^+_i$ (resp. $\gamma^-_j$). The standard Riemann-Roch formula then gives rise to
the formula for the Euler characteristic
\beastar
\chi(D) & = & 2\dim_\C H^0(D) - 2\dim_\C H^1(D) = 2\operatorname{deg} (D) - 2g\\
& = &\sum_{i=1}^{s^+} 2m(\gamma^+_i)+ \sum_{j=1}^{s^-} 2m(\gamma^-_j) - 2g.
\eeastar

This finishes the proof.

\end{proof}

\section{Generic transversality under the perturbation of $J$}
\label{sec:generic}

We start with recalling the linearization of the equation $\dot x  = R_\lambda(x)$
along a closed Reeb orbit. Let $z$ be a closed Reeb orbit of period $T > 0$. In other words,
$z: \R \to Q$ is a periodic solution of $\dot z = R_\lambda(z)$ with period $T$, thus satisfying $z(T) = z(0)$.

Denote by $\phi^t= \phi^t_{R_\lambda}$ the flow of the Reeb vector field $R_\lambda$,
so that we can write $z(t) = \phi^t(z(0))$.
In particular $p:= z(0)$ is a fixed point of the diffeomorphism $\phi^T$.
Further, since $L_{R_\lambda}\lambda = 0$,  the contact diffeomorphism $\phi^T$ induces the isomorphism
$$
\Psi_z : = d\phi^T(p)|_{\xi_p}: \xi_p \to \xi_p
$$
which is the tangent map of the Poincar\'e return map $\phi^T$ restricted to $\xi_p$.

\begin{defn} We say a Reeb orbit with period $T$ is \emph{nondegenerate}
if $\Psi_z:\xi_p \to \xi_p$ with $p = z(0)$ has no eigenvalue 1.
\end{defn}

Denote by $\Cont(Q,\xi)$ the set of contact forms with respect to the contact structure $\xi$ and $\CL(Q)=C^\infty(S^1,Q)$
the space of loops $z: S^1 = \R /\Z \to Q$. Let $\CL^{1,2}(Q)$ be the $W^{1,2}$-completion of $\CL(Q)$.
We would like to consider some Banach vector bundle $\CL$ over the Banach manifold
$(0,\infty) \times \CL^{1,2}(Q) \times \Cont(Q,\xi)$ whose fiber at $(T, z, \lambda)$ is given by $L^2(z^*TQ)$.
We consider the assignment
$$
\Upsilon: (T,z,\lambda) \mapsto \dot z - T \,R_\lambda(z)
$$
which is a section of $\CL$.

Denote by
$$
DR_\lambda(z)
$$
the covariant derivative of $R_\lambda$ along the
curve $z$. Then we have the following expression of the full linearization.
\begin{lem}\label{lem:full-linearization}
\beastar
d{(T,z, \lambda)}\Upsilon(a,Y, B) = \frac{D Y}{dt} - T D R_\lambda(z)(Y)-a R_\lambda- T \delta_{\lambda}R_\lambda(B),
\eeastar
where $a\in \R$, $Y\in T_z\CL^{1,2}(Q)=W^{1,2}(z^*TQ)$ and $B\in T_\lambda \Cont(Q, \xi)$ and
the last term $\delta_{\lambda}R_\lambda$ is some linear operator.
\end{lem}

By using this full linearization, one can study the generic existence of the contact one-forms which make all
Reeb orbits nondegenerate. We refer to Appendix of \cite{wendl} for its complete proof.
We assume in the rest of our discussion below that $\lambda$ is such a generic contact form.

Now we involve the set $\CJ(Q \times S^1,\lambda)$ given in Definition \ref{defn:lambda-admissible-J}.
We study the linearization of the map $\Upsilon^{univ}$ which is the map $\Upsilon$ augmented by
the argument $J \in \CJ(Q \times S^1,\lambda)$. More precisely, we define
$$
\Upsilon^{univ}(j, w, J) = \left(\delbar_J^\pi w, w^*\lambda \circ j - f^*d\theta\right)
$$
$\delbar$ at each $(j,w,J) \in \delbar^{-1}(0)$. In the discussion below, we will fix the complex
structure $j$ on $\Sigma$, and so suppress $j$ from the argument of $\Upsilon^{univ}$.

We denote the zero set $(\Upsilon^{univ})^{-1}(0)$ by
$$
\MM(\dot \Sigma,Q \times S^1;\vec \gamma^+, \vec \gamma^-) = \left\{ (w,J)
|\, \Upsilon^{univ}(w, J) = 0 \right\}
$$
which we call the universal moduli space, where
$$
(w,J) \in \CW^{k,p}_\delta(\dot \Sigma,Q \times S^1;\vec \gamma^+, \vec \gamma^-) \times \CJ^\ell(Q \times S^1,\lambda).
$$
Denote by
$$
\pi_2: \CW^{k,p}_\delta(\dot \Sigma,Q \times S^1;\vec \gamma^+, \vec \gamma^-) \times \CJ^\ell(Q \times S^1,\lambda) \to
\CJ^\ell(Q \times S^1,\lambda)
$$
the projection. Then we have
\be\label{eq:MMK}
\MM_J(\dot \Sigma,Q \times S^1;\vec \gamma^+, \vec \gamma^-)
 = \pi_2^{-1}(J) \cap \MM(\dot \Sigma,Q \times S^1;\vec \gamma^+, \vec \gamma^-).
\ee

One essential ingredient for the generic transversality under the perturbation of
$J \in \CJ(Q \times S^1,\lambda)$ is the usage of the unique continuation result,
which applies to arbitrary $J$-holomorphic curves on almost complex manifolds.

\begin{prop}[Unique continuation lemma]\label{prop:unique-conti}
Any non-constant $\lcs$ instanton does not
have an accumulation point in the zero set of $dw$.
\end{prop}

\begin{rem} In fact this unique continuation applies to general
contact instanton $w$ whose proof strongly relies on
the closedness of the one-form $w^*\lambda \circ j$.
(See \cite{oh:contacton} for its proof.)
\end{rem}

The following theorem summarizes the main transversality scheme
needed for the study of the moduli problem of the contact instanton map, whose proof
is not very different from that of pseudo-holomorphic curves, once the above
unique continuation result is established, and so is omitted.

\begin{thm}\label{thm:trans} Let $0 < \ell < k -\frac{2}{p}$.
Consider the moduli space $\MM(\dot \Sigma,Q \times S^1;\vec \gamma^+, \vec \gamma^-)$. Then
\begin{enumerate}
\item $\MM(\dot \Sigma,Q \times S^1;\vec \gamma^+, \vec \gamma^-)$ is
an infinite dimensional $C^\ell$ Banach manifold.
\item The projection $\Pi_\alpha =
\pi_2|_{\MM(\dot \Sigma,Q \times S^1;\vec \gamma^+, \vec \gamma^-)}:
\MM(\dot \Sigma,Q \times S^1;\vec \gamma^+, \vec \gamma^-) \to \CJ^\ell(Q \times S^1,\lambda)$ is a
Fredholm map and its index is the same as that of $D\Upsilon(w)$
for a (and so any) $w \in  \MM(Q,\lambda,J;\overline \gamma, \underline \gamma)$.
\end{enumerate}
\end{thm}

One should compare this with the corresponding statement for Floer's perturbed Cauchy-Riemann equations
on symplectic manifolds.

\appendix

\section{Proof of energy bound for the case of proper potential}
\label{sec:energy-bound}

In this appendix, we give the proof of Proposition \ref{prop:uniform-energy-bound}.

Since $f$ is assumed to be proper, $f(r) = \pm \infty$
for each puncture $r_\ell$ of $\dot \Sigma$ depending on whether the puncture is positive or
negative.

The proof is entirely similar to the proof of Lemma 5.15 \cite{behwz} verbatim
with replacement of $a$ and the equation $dw^*\lambda \circ j = da$ therein by $f$
and the equation
$$
dw^*\lambda \circ j + \sum_{e \in E(T)} Q(w;e)\, dt_e = df
$$
respectively in our current context. (\emph{We would also like to point out that \cite{behwz}
used the letter `$f$' for the map $w$ while our notation $f$ is for the contact instanton potential function
 which corresponds to $a$ in their notation}. This should not confuse the readers, hopefully.)

In a neighborhood $D_\delta(p) \subset \C$ of a given puncture $p$ with analytic coordinate
$z$ centered at $p$ and $C_\delta(p) = \del D_\delta(p)$, oriented positively
for a positive puncture, and negatively for a
negative puncture, consider the function
$$
\delta \mapsto \int_{C_\delta(p)} w^*\lambda.
$$
It is increasing and bounded above (resp. decreasing and bounded below), if the puncture is
positive (resp. negative), since $d\lambda \geq 0$ on any contact Cauchy-Riemann map $w$ and
$\int_{D_\delta(p)} dw^*\lambda \leq E^\pi(w) < \infty$. Therefore the integral
$$
\int_{C_\delta(p)} w^*\lambda
$$
has a finite limit as $\delta \to 0$ for all punctures. Now let $\varphi \in \CC$ and let $\varphi_n \in \CC$ such that
$\|\varphi - \varphi_n\|_{C^0} \to 0$ and
$\varphi_n\circ f = 0$ on $D_{\frac{1}{n}}(p)$ for all punctures $p$. Such functions exist by
the assumption on properness of the potential function $f$. Moreover we can choose $\varphi_n$ so that
$$
\int_{\dot\Sigma}(\varphi_n \circ f)\, df \wedge w^*\lambda = \int_{\dot \Sigma} w^*d(\psi_n w^*\lambda)
- \int_{\dot \Sigma} (\psi_n \circ f) w^*d\lambda,
$$
where $\psi_n(s) = \int_{-\infty}^s \varphi_n(\sigma)\, d\sigma$. Notice that $\psi_n\circ f = 1$ in
$D_{\frac{1}{n}}(p)$ when $p$ is a positive puncture and $\psi_n \circ f = 0$ therein when $p$ is negative.
By Stokes' theorem,
$$
\int_{\dot \Sigma} w^*d(\psi_n \lambda) = \lim_{\delta \to 0} \sum_{\ell^+}
\int_{\del_{\ell^+} \Sigma(\delta)} w^*\lambda
$$
where the sum is taken over all positive punctures $p_{\ell^+}$. Therefore
\beastar
\int_{\dot \Sigma} (\varphi_n \circ f)\, df \wedge w^*\lambda & = & \lim_{\delta \to 0}
\sum_{\ell^+} \int_{\del_{\ell^+} \Sigma(\delta)} w^*\lambda - \int_{\dot \Sigma}(\psi_n \circ f)\, w^*d\lambda \\
& \leq & \lim_{\delta \to 0} \sum_{\ell^+}\int_{\del_{\ell^+}D_\delta(p)} w^*\lambda < C' < \infty.
\eeastar
Moreover
$$
\int_{\dot \Sigma} (\varphi_n \circ f)\, df \wedge w^*\lambda \rightarrow \int_{\dot \Sigma}
(\varphi\circ f)\, df \wedge w^*\lambda
$$
as $n \to \infty$, which implies
$$
\int_{\dot \Sigma} (\varphi \circ f)\, df  \leq C',
$$
and so $E(w) \leq E^\pi(w) + C' < \infty$. This finishes the proof.

\end{document}